\documentclass[centering,11pt,reqno]{amsart}

\usepackage[utf8]{inputenc}
\usepackage[T1]{fontenc}
\usepackage{amsmath,amsthm}
\usepackage{amsfonts,amssymb}
\usepackage[mathscr]{eucal}
\usepackage{url}
\usepackage{paralist}
\usepackage[colorlinks,cite color=blue,pagebackref=true,pdftex]{hyperref}
\usepackage[margin=2.5cm]{geometry}
\usepackage{tikz}
\usetikzlibrary{calc,shapes}
\usepackage{mathtools}
\usepackage{multicol}
\usepackage{ytableau}
\usepackage{blkarray}
\usepackage[capitalize]{cleveref}
 \usepackage[foot]{amsaddr}
\usepackage{tikz-cd}
\usepackage{ytableau}
\usepackage{bbm}
\usepackage{comment}
\usepackage{color}
\usepackage{todonotes}
\usepackage{accsupp}

\newtheorem{theorem}{Theorem}[section]
\newtheorem*{theorem*}{Theorem}
\newtheorem*{cor*}{Corollary}
\newtheorem*{conj*}{Conjecture}
\newtheorem*{lemma*}{Lemma}
\newtheorem*{prop*}{Proposition}
\newtheorem{lem}[theorem]{Lemma}
\newtheorem{cor}[theorem]{Corollary}
\newtheorem{prop}[theorem]{Proposition}
\newtheorem{conj}[theorem]{Conjecture}

\theoremstyle{definition}
\newtheorem{example}[theorem]{Example}

\newtheorem{definition}[theorem]{Definition}

\newtheorem{proposition}[theorem]{Proposition}

\renewcommand{\emptyset}{\varnothing}
\renewcommand{\epsilon}{\varepsilon}

\newcommand{\dg}{\mathrm{dg}}
\newcommand{\conv}{\mathrm{conv}}

\newcommand{\Newt}{\mathrm{Newton}}
\newcommand{\cL}{\mathcal{L}}
\newcommand{\supp}{\mathrm{supp}}
\newcommand{\mult}{\mathrm{mult}}

\title{Saturation for Non-Symmetric Macdonald Polynomials}

\author[Bechtloff Weising]{Milo Bechtloff Weising}
\address{Department of Mathematics, Virginia Tech,
Blacksburg, VA 24061}
\email{milojbw@vt.edu}
\author[Black]{Alexander E. Black}
\address{Department of Mathematics, Bowdoin College, Brunswick, ME 04011}
\email{a.black@bowdoin.edu}

\begin{document}

\maketitle

\begin{abstract}  
We prove that supports of non-symmetric Macdonald polynomials are $M$-convex. As a consequence, we resolve a 2019 conjecture of Monical, Tokcan, and Yong that they have the saturated Newton polytope property. As a corollary we show that affine Demazure characters of type $\mathrm{GL}$ have M-convex supports and therefore the saturated Newton polytope property answering a 2022 open question of Besson and Hong. By their results, we then find that certain affine analogs of Bruhat interval polytopes in type $\mathrm{GL}$ are generalized permutahedra. To prove these results, we find a novel interpretation of the Haglund--Haiman--Loehr formula for non-symmetric Macdonald polynomials in terms of colorings of Dyck graphs. 
\end{abstract}

\section{Introduction}

Understanding the coefficients of polynomials from  representation theory and algebraic geometry is a  central goal of algebraic combinatorics. Such polynomials tend to be defined implicitly as the result of a sequence of applications of operators such as divided difference operators and their variants or as a class in cohomology or K-theory. Even determining which coefficients of these polynomials are nonzero is nontrivial and carries meaning such as whether a certain irreducible representation appears in an isotypic decomposition. An approach to this problem explored in depth in \cite{SNPOriginal} is to give a polyhedral characterization of the monomials with nonzero coefficient. Namely, given a polynomial $f \in \mathbb{R}[x_{1}, x_{2}, \dots, x_{n}]$ such that $f(x) = \sum_{ \alpha \in \mathbb{Z}_{\geq 0}^{n}} c_{\alpha} x^{\alpha}$ with $x^{\alpha} = x_{1}^{\alpha_{1}}\dots x_{n}^{\alpha_{n}}$, one may compute its Newton polytope given by $\Newt(f) = \text{conv}(\{\alpha \in \mathbb{Z}^{n}: c_{\alpha} \neq 0\})$. Each monomial corresponds to a lattice point in that polytope. We say that $f$ has a \textbf{saturated Newton polytope} or is \textbf{saturated} if the converse is true: 
\[\Newt(f) \cap \mathbb{Z}^{n} = \{\alpha \in \mathbb{Z}^{n}: c_{\alpha} \neq 0\}.\]
Equivalently, a polynomial is saturated if a monomial appears with nonzero coefficient if and only if its exponent vector satisfies a system of linear inequalities. A first example of this phenomenon is the determinant. Its Newton polytope is the Birkhoff polytope, and the only integer points it contains are precisely its vertices. Hence, the determinant has a saturated Newton polytope. Any square free polynomial is more generally saturated, but the phenomenon becomes rarer and harder to verify in general.

Saturation gained attention recently as it holds for Lorentzian polynomials \cite{LorentzianOrig}. These polynomials appear in connection to the Kahler package of Adiprasito, Huh, and Katz \cite{AHKBreakthrough} used to resolve longstanding conjectures in matroid theory related to log-concavity, as well as in the theory of real stability of polynomials used to establish fast mixing for random walks on basis exchange graphs of matroids \cite{LogConcavePoly1}. Lorentzian polynomials have a stronger property than saturation. Their supports are \textbf{M-convex sets}, meaning that they have saturated Newton polytopes, and their Newton polytope is a generalized permutahedron in the sense of Postnikov \cite{GenPermOrig}. In fact, it was shown in \cite{LorentzianOrig} that a polynomial has M-convex support if and only if it has the same support set as some Lorentzian polynomial. Proving the support is M-convex has more utility than saturation on its own as M-convex sets are closed under Minkowski sums, which implies that polynomials with positive coefficients and M-convex support are closed under multiplication. This property is not true for saturation in general, where the square of a polynomial with saturated Newton polytope need not still be saturated.   

In \cite{SNPOriginal}, Monical, Tokcan, and Yong observed that many polynomials appearing in algebraic combinatorics have saturated Newton polytopes. They proved this, for example, for Schur polynomials, symmetric and modified Macdonald polynomials, elementary symmetric functions, and resultants among several other examples. They also provided non-examples such as discriminants. Based on computational evidence from small examples, they conjectured Schubert polynomials, double Schubert polynomials, key polynomials, Grothendieck polynomials, Kronecker products of Schur polynomials, Lascoux polynomials, Demazure atoms, and non-symmetric Macdonald polynomials were all saturated. 

Later, Fink, M{\'e}sz{\'a}ros, and St. Dizier proved that Schubert polynomials and key polynomials had the saturated Newton polytope property via an exact computation of their Newton polytopes as a Minkowski sum of matroid base polytopes \cite{SchubertSNP}. Besson, Jeralds, and Kiers later extended the results for key polynomials with an alternative proof for more general Coxeter type \cite{KeySNPGeneralType}, and Castillo, Cid Ruiz, Mohammadi, and Monta\~{n}o showed saturation for double Schubert polynomials using the theory of multidegrees from commutative algebra \cite{DoubleSchubertSNP}. For Grothendieck polynomials and Kronecker products of Schur polynomials, several special cases of the conjectures are known to hold \cite{VexillaryGrothendieck, Schubert01Grothendieck, FlowGrothendieck,SymmetricGrothendieck, FireworksGrothendieck, KroneckerProduct, HafnerCMpolynomials}. Beyond those conjectures, in connection to the theory of Lorentzian polynomials, several other families of polynomials from algebraic combinatorics are known to have saturated Newton polytopes \cite{ChromaticSNP, LogConcaveManyPoly, AlexanderPolynomial, DualKSchur, FPolynomials, ClusterPolynomials, PostnikovStanley, EquivariantCohomology, GoodSymmetric, SuperSchur}. For Lascoux polynomials, Demazure atoms, and non-symmetric Macdonald polynomials, to our knowledge, nothing was known prior to this work even in special cases. Here we resolve the following case equivalent to Conjecture $3.8$ of \cite{SNPOriginal}:

\begin{theorem}
\label{thm:MacdonaldSNP}
The supports of non-symmetric Macdonald polynomials are M-convex. In particular, they have the saturated Newton polytope property.
\end{theorem}

The \textbf{non-symmetric Macdonald polynomials} $\{E_{\mu}(x_1,\dots,x_n)| \mu \in \mathbb{Z}^n\}$ form an exceptional basis for the space of Laurent polynomials $\mathbb{Q}(q,t)[x_1^{\pm 1},\dots,x_n^{\pm 1}]$ occurring across modern algebraic combinatorics in the geometry of Hilbert schemes \cite{haiman2000hilbert}, the combinatorics of LLT polynomials \cite{FlaggedLLT, HHLsym, HHLnsym}, and the study of double affine Hecke algebras \cite{C_95}. They appear in statistical mechanics, which provides probabilistic interpretations of their formulas using solvable lattice models \cite{IntegrableVertexModels} or multiline queues \cite{MultilineQueues}. They also relate to other families of polynomials via specialization of coefficients such as key polynomials \cite{keyfromMacdonald}, Demazure atoms \cite{DemazureAtoms}, and affine Demazure characters of type $\mathrm{GL}$ \cite{Kumar02}. By Theorem 8.2.2 of \cite{Kumar02} non-symmetric Macdonald polynomials have the same support as affine Demazure characters of type $\mathrm{GL}$. Relatedly, in Section 1.2 of \cite{AffineDems}, Besson, Jeralds, and Kiers recently asked whether affine Demazure characters have saturated Newton polytopes. This question was also raised earlier at the conclusion of \cite{BessonHong} by Besson and Hong in terms of weight polytopes and toric varieties in affine Schubert varieties, where they had observed computationally this property held in small cases. We resolve the type $\mathrm{GL}$ case here:

\begin{cor}
    For each $\mu \in \mathbb{Z}^n_{\geq 0}$ the type $\mathrm{GL}_n$ affine Demazure character $E_{\mu}(x;q,0)$ has M-convex support for generic $q.$ In particular, they have the saturated Newton polytope property. 
\end{cor}

By Sahi's triangularity condition from \cite{Sahi}, the support of the non-symmetric Macdonald polynomials  $E_{\mu}$ is precisely the lower interval with maximum $\mu$ in the affine type $\mathrm{GL}$ Bruhat order.   Besson and Hong noted in \cite{BessonHong} that the moment polytopes of \textbf{affine Schubert varieties} in affine Grassmannians are precisely the convex hulls of certain lower order ideals in the Bruhat order of the co-weight lattice. Thus, these moment polytopes are the Newton polytopes of non-symmetric Macdonald polynomials. As a consequence of Theorem \ref{thm:MacdonaldSNP}, we find the following.

\begin{cor}
\label{cor: schubert mp}
    For all $\mu \in \mathbb{Z}^n,$ the moment polytope $\mathrm{MP}(\overline{X}_{\mu})$ of the affine Schubert variety $\overline{X}_{\mu} \subset \mathrm{Gr}_n$ is a generalized permutahedron. 
\end{cor}

These polytopes are affine analogs of the Bruhat interval polytopes introduced by Tsukerman and Williams in \cite{BruhatIntervalPolytopesOrig}.  In Proposition 2.7 of \cite{BruhatIntervalPolytopesOrig}, they showed that Bruhat interval polytopes are generalized permutahedra. Corollary \ref{cor: schubert mp} shows this remains true in the affine type $\mathrm{GL}$ setting for lower intervals. 

To prove these results, we give a novel interpretation of the Haglund--Haiman--Loehr (HHL) combinatorial formula for non-symmetric Macdonald polynomials from \cite{HHLnsym} in terms of graph coloring. In particular, given any Dyck path, one can associate a graph to it called its Dyck graph. The chromatic symmetric function of a Dyck graph is a fundamental object in algebraic combinatorics due to its role in the theory of Hessenberg varieties \cite{ChromaticHessenberg} most well known in the context of the Stanley-Stembridge conjecture recently resolved by Hikita \cite{hikitastanleystembridge} with a different proof given shortly after in \cite{macdonaldstanleystembridge}. The latter proof used a connection between Macdonald polynomials and Dyck graphs unearthed and leveraged in the proof of the Shuffle Theorem \cite{ShuffleConjecture} and explored further in \cite{LLTChromatic, ChromaticMacdonald}. In \cite{ChromaticMacdonald}, they derived decompositions for Macdonald polynomials into chromatic quasisymmetric polynomials of certain Dyck graphs and introduced a non-symmetric analog of the chromatic quasisymmetric polynomial to generalize this decomposition to the non-symmetric setting. In that work, they introduced a special family of Dyck graphs called attacking graphs and claimed without proof that attacking graphs are Dyck graphs. 

In Section \ref{sec:HHL}, we describe how to interpret the supports of non-symmetric Macdonald polynomials using the HHL formula, and in Section \ref{sec:chromatic}, we prove that non-symmetric Macdonald polynomials have the same support as a distinct non-symmetric generalization of the chromatic symmetric polynomial of Dyck graphs from that of \cite{ChromaticNonsym} given by only looking at colorings that extend a fixed coloring of a subgraph. Initially, we thought this approach would directly prove saturation, since in \cite[Theorem 4.1]{ChromaticSNP} they showed that chromatic symmetric polynomials of Dyck graphs have saturated Newton polytopes. However, the lack of symmetry obstructs such an extension.

Nonetheless our proof relies on our chromatic reinterpretation of the HHL combinatorial formula for non-symmetric Macdonald polynomials \cite{HHLnsym}, but it is not a purely combinatorial argument as in \cite{ChromaticSNP}. Our argument requires leveraging the algebraic structure of non-symmetric Macdonald polynomials. Namely, we prove our result inductively by noting that the trivial non-symmetric Macdonald polynomial $E_{(0,\dots,0)} = 1$ has M-convex support and any non-symmetric Macdonald polynomial may be built from smaller ones starting with $E_{(0,\dots,0)} = 1$ by applying a sequence of applications of two operations from the Knop-Sahi relations. Using the graph coloring interpretation, we interpret how these symmetries affect the set of colorings in Section \ref{sec:symmetries}. Then we require a novel polyhedral geometry argument in Section \ref{sec:MConvexity} of independent interest to show that both operations preserve M-convexity of the support sets of the polynomials.

In Section \ref{recoveringNonSym}, we discuss our reinterpretation of the HHL formula for non-symmetric Macdonald polynomials in terms of colorings of Dyck graphs and derive an analogous interpretation for modified Macdonald functions. The combinatorial statistics for the HHL formulas naturally extend to colorings of arbitrary Dyck graphs. Thus we are able to define analogous constructions for arbitrary Dyck paths. In the modified case, we observe that the resulting power series are no longer always symmetric but are quasi-symmetric and, in fact, have positive expansions into Gessel's fundamental quasi-symmetric functions.


\subsection{Funding}
The authors have no funding sources to report.

\subsection{Acknowledgments}

The authors would like to thank Sean Griffin, Joseph Pappe, and Germain Poullot for useful discussions during the FPSAC 2025 conference at Hokkaido University. We thank Karola M{\'e}sz{\'a}ros for suggesting checking the (denormalized) Lorentzian property and log-concavity.  

\section{Non-symmetric Macdonald polynomials and their combinatorics}
\label{sec:HHL}
\subsection{Non-symmetric Macdonald polynomials}

The non-symmetric Macdonald polynomials $E_{\mu}$ may be defined in multiple equivalent ways. We refer the reader to \cite{HHLnsym}, \cite{MultilineQueues}, and \cite{C_95} for an overview of these various definitions. We use the HHL combinatorial description of non-symmetric Macdonald polynomials (Theorem 3.5.1 \cite{HHLnsym}) throughout this paper and thus take their combinatorial formula as a definition. However, before we may define the non-symmetric Macdonald polynomials $E_{\mu}(x_1,\dots,x_n;q,t)$, we must first define some required combinatorial objects.

\begin{definition}\cite{HHLnsym} \label{HHL defn}
For a weak composition $\mu = (\mu_1,\dots,\mu_n)$, define the column diagram of $\mu$ as 
$$dg'(\mu):= \{(i,j)\in \mathbb{N}^2 : 1\leq i\leq n, 1\leq j \leq \mu_i \}.$$ This is represented by a collection of boxes in positions given by $dg'(\mu)$. The augmented diagram of $\mu$ is given by 
$$\widehat{dg}(\mu):= dg'(\mu)\cup\{(i,0): 1\leq i\leq n\}.$$
Visually, to get $\widehat{dg}(\mu)$ we are adding a bottom row of boxes on length $n$ below the diagram $dg'(\mu)$. 

A filling of $\mu$ is a function $\sigma: dg'(\mu) \rightarrow \{1,\dots,n\}$ and given a filling there is an associated augmented filling $\widehat{\sigma}: \widehat{dg}(\mu) \rightarrow \{1,\dots,n\}$ extending $\sigma$ with the additional bottom row boxes filled according to $\widehat{\sigma}((j,0)) = j$ for $j = 1,\dots,n$. Distinct lattice squares $u,v \in \mathbb{N}^2$ are said to attack each other if one of the following is true:
\begin{itemize}
\item $u$ and $v$ are in the same row 
\item $u$ and $v$ are in consecutive rows and the box in the lower row is to the right of the box in the upper row.
\end{itemize}
A filling $\sigma: dg'(\mu) \rightarrow \{1,\dots,n\}$ is \textbf{\textit{non-attacking}} if $\widehat{\sigma}(u) \neq \widehat{\sigma}(v)$ for every pair of attacking boxes $u,v \in \widehat{dg}(\mu).$ For a filling $\sigma:\dg'(\mu) \rightarrow \{1,\dots,n\}$ set $$x^{\sigma}:= x_1^{|\sigma^{-1}(1)|}\cdots x_n^{|\sigma^{-1}(n)|}.$$
\end{definition}

\begin{example}

Below is a lattice diagram where we have chosen some box $x$ and labeled all of the boxes with $\ast$'s which either attack $x$ or are attacked by $x:$

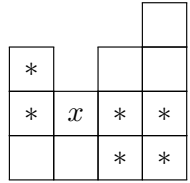
\begin{center}
    \begin{ytableau}
        \none & \none    & \none &  \\
           \ast   &  \none   &       &  \\
           \ast   &     x     &   \ast    & \ast \\
              &          &    \ast  & \ast \\
    \end{ytableau}
    \end{center}
    
Any box is only either attacked by or attacks a box in adjacent rows and, in particular, the only boxes of $\dg'(\mu)$ which attack the basement boxes are those in the first row. 
\end{example}

\begin{example}
    Consider $\mu = (1,0,2).$ The diagram $\widehat{\dg}(\mu)$ is given by
    
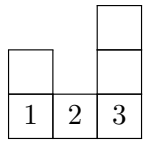
\begin{center}
    \begin{ytableau}
        \none & \none &  \\
         &   \none   &   \\
        1 & 2 & 3 \\
    \end{ytableau}
    \end{center}
    where the boxes labeled $1,2,3$ are the basement of $\widehat{\dg}(\mu)$. By definition, these basement boxes always carry the labels $1,2,3$ in any non-attacking labeling of $\widehat{\dg}(\mu)$. The following are valid non-attacking labelings $\widehat{\sigma}$ of $\widehat{\dg}(\mu)$: 
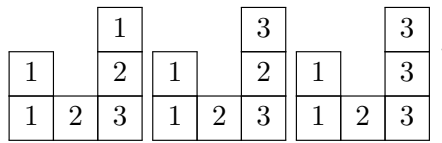
    \begin{center}
        \begin{ytableau}
        \none & \none & 1 \\
        1 &   \none   & 2  \\
        1 & 2 & 3 \\
    \end{ytableau} ~~~ \hspace{.1in} \begin{ytableau}
        \none & \none & 3 \\
        1 &   \none   & 2  \\
        1 & 2 & 3 \\
    \end{ytableau} ~~~ \hspace{.1in} \begin{ytableau}
        \none & \none & 3 \\
        1 &   \none   & 3  \\
        1 & 2 & 3 \\
    \end{ytableau} .
    \end{center}

    For example, the following labelings are invalid as they all have a pair of attacking boxes with the same label, which we mark with $*$'s:
    \begin{center}
    \begin{ytableau}
        \none & \none & 1 \\
        2* &   \none   & 1  \\
        1 & 2* & 3 \\ 
    \end{ytableau} ~~~\hspace{.1in} \begin{ytableau}
        \none & \none & 2 \\
        1* &   \none   & 1*  \\
        1 & 2 & 3 \\ 
    \end{ytableau} ~~~ \hspace{.1in} \begin{ytableau}
        \none & \none & 2 \\
        3* &   \none   & 1  \\
        1 & 2 & 3* \\ 
    \end{ytableau}.
    \end{center}

\end{example}

 The combinatorial formula for non-symmetric Macdonald polynomials can now be stated. The coefficients in the below formula involve various combinatorial statistics, which we defer defining to Section \ref{recoveringNonSym}. However, as we consider \textbf{generic} $q,t$ the values of these statistics is irrelevant for describing the supports outside of knowing that they are finite non-negative integers.

\begin{definition}\label{def: nsym macd}\cite{HHLnsym}
For $\mu \in \mathbb{Z}_{\geq 0}^{n}$ define the non-symmetric Macdonald polynomial $E_{\mu} \in \mathbb{Q}(q,t)[x_1,\dots,x_n]$ as
    $$E_{\mu} := \sum_{\substack{\sigma: \mu \rightarrow \{1,\dots,n\}\\ \text{non-attacking}}} x^{\sigma}q^{\mathrm{maj}(\widehat{\sigma})}t^{\mathrm{coinv}(\widehat{\sigma})} \prod_{\substack{\square \in \dg'(\mu) \\ \widehat{\sigma}(\square) \neq \widehat{\sigma}(d(\square))}} \left( \frac{1-t}{1-q^{l(\square)+1}t^{a(\square)+1}} \right). $$ 
    Here $\sigma:\mu \rightarrow \{1,\dots,n\}$ is shorthand for diagram labelings of $\dg'(\mu).$ For general $\mu \in \mathbb{Z}^n$, say $\mu = (m,\dots,m) + \mu'$ with $\mu' \in \mathbb{Z}_{\geq 0}^n$ and $m \in \mathbb{Z},$ define $E_{\mu}:= (x_1\cdots x_n)^m E_{\mu'}.$
\end{definition}

\begin{example}
    For $\mu = (0,2,0)$, the non-symmetric Macdonald polynomial $E_{(0,2,0)}(x_1,x_2,x_3;q,t)$ is given by 
    $$E_{(0,2,0)} = x_2^{2} + \frac{1-t}{1-q^2t^2} x_1^2 + \frac{q(1-t)}{1-qt} x_2x_3 + \frac{q(1-t)^2}{(1-qt)(1-q^2t^2)} x_1x_3 + \frac{1-t}{1-qt} x_1x_2 + \frac{q(1-t)^2}{(1-qt)(1-q^2t^2)} x_1x_2.$$ Each of the terms above corresponds to exactly one of the following non-attacking labelings:
    
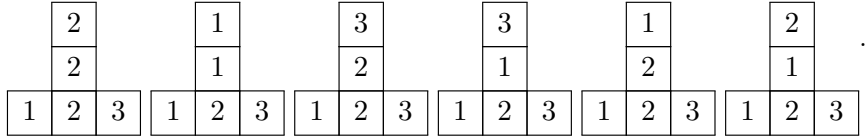
\begin{center}
        \begin{ytableau}
        \none & 2 & \none \\
        \none & 2 & \none  \\
        1 & 2 & 3 \\ 
    \end{ytableau} ~~~~ \hspace{.1in} \begin{ytableau}
        \none & 1 & \none \\
        \none & 1 & \none  \\
        1 & 2 & 3 \\ 
    \end{ytableau}~~~~ \hspace{.1in} \begin{ytableau}
        \none & 3 & \none \\
        \none & 2 & \none  \\
        1 & 2 & 3 \\ 
    \end{ytableau}~~~~ \hspace{.1in} \begin{ytableau}
        \none & 3 & \none \\
        \none & 1 & \none  \\
        1 & 2 & 3 \\ 
    \end{ytableau}~~~~ \hspace{.1in} \begin{ytableau}
        \none & 1 & \none \\
        \none & 2 & \none  \\
        1 & 2 & 3 \\ 
    \end{ytableau}~~~~ \hspace{.1in} \begin{ytableau}
        \none & 2 & \none \\
        \none & 1 & \none  \\
        1 & 2 & 3 \\ 
    \end{ytableau}.
    \end{center}
\end{example}

Since whenever $\mu = (m,\dots,m) + \mu'$ with $\mu' \in \mathbb{Z}_{\geq 0}^n,$ $\mathrm{Newton}(E_{\mu}) = (m,\dots,m) + \mathrm{Newton}(E_{\mu'})$, it suffices for our purposes throughout the remainder of this paper to only consider those $E_{\mu'}$ for $\mu' \in \mathbb{Z}_{\geq 0}^{n}.$ Given a weak composition $\mu = (\mu_1,\dots,\mu_n)$ write $|\mu|:= \mu_1+\dots + \mu_n.$

The following operations are important in this paper.

\begin{definition}
    For $1\leq i \leq n-1$ define $s_i: \mathbb{Z}^n \rightarrow \mathbb{Z}^n$ as the simple transposition $s_i(\dots,a_i,a_{i+1},\dots):= (\dots,a_{i+1},a_{i},\dots).$ Define the map $\pi: \mathbb{Z}^n \rightarrow \mathbb{Z}^n$ as $\pi(a_1,\dots,a_n):= (a_n+1,a_1,\dots,a_{n-1}).$ Define $\psi: \mathbb{Q}(q,t)[x_1^{\pm 1},\dots, x_n^{\pm 1}] \rightarrow \mathbb{Q}(q,t)[x_1^{\pm 1},\dots, x_n^{\pm 1}]$ for the $\mathbb{Q}(q,t)$-algebra automorphism determined by 
    $\psi(x_1^{a_1}\cdots x_n^{a_n}):= q^{-a_n} x_1^{a_n}x_2^{a_1}\cdots x_n^{a_{n-1}}.$
\end{definition}

Notice that the actions of $s_1,\dots,s_{n-1},\pi$ on $\mathbb{Z}^n$ generate the standard action of the extended affine symmetric group $\widehat{\mathfrak{S}}_n$ on $\mathbb{Z}^n.$ Furthermore, $s_1,\dots,s_{n-1},\pi$ restrict to maps $\mathbb{Z}^{n}_{\geq 0} \rightarrow \mathbb{Z}^{n}_{\geq 0}$, and $s_1,\dots,s_{n-1}$ are invertible whereas $\pi \mid_{\mathbb{Z}^{n}_{\geq 0}}$ is now only injective and no longer invertible. We write $e_1,\dots,e_n$ for the standard coordinate basis vectors of $\mathbb{Z}^n.$

The following is one of the Knop-Sahi recurrence relations for the non-symmetric Macdonald polynomials, which we require later:

\begin{prop}\label{prop:Knop-Sahi}\cite{HHLnsym}
    For all $\mu \in \mathbb{Z}^n$, 
    $E_{\pi(\mu)} = q^{\mu_n} x_1\psi (E_{\mu}).$
\end{prop}

\subsection{Newton polytopes}

\begin{definition}
    Given a Laurent polynomial $f(x_1,\dots,x_n) \in \mathbb{R}[x_1^{\pm 1},\dots, x_n^{\pm 1}]$ and $\alpha \in \mathbb{Z}^n$, let $\langle x^{\alpha} \rangle f$ denote the coefficient of $x^{\alpha}$ in the monomial expansion of $f.$ The set of all $\alpha \in \mathbb{Z}^n$ such that $\langle x^{\alpha} \rangle f \neq 0$ is called the \textbf{support} set of $f$ and denoted $\mathrm{supp}(f).$ The \textbf{Newton polytope} $\Newt(f)$ is the polytope $\mathrm{conv}( \mathrm{supp}(f)).$ The polynomial $f$ is said to have the \textbf{saturated Newton polytope property} (SNP) if  $\Newt(f) \cap \mathbb{Z}^n = \mathrm{supp}(f).$
\end{definition}

We are interested in studying the Newton polytopes of the non-symmetric Macdonald polynomials $\Newt(E_{\mu}) = \Newt(E_{\mu}(x_1,\dots,x_n;q,t))$ with $q,t$ \textit{generic}. For the remainder of the paper, if we refer to $E_{\mu}$ without specifying $q,t$ explicitly we are taking $q,t$ generic. If the reader wishes to avoid the term generic, then one may replace the condition that $q,t$ be generic with any choice $0< q,t < 1.$ In particular, from the explicit formula in Definition \ref{def: nsym macd} it is clear that $q = t = 1/2$ is always sufficiently generic.

\section{From Non-symmetric Macdonald Polynomials to Graph Coloring}

\label{sec:chromatic}
 
 From the augmented diagram, we can obtain a graph $\Gamma_{\mu}$ with vertex set $\widehat{dg}(\mu)$ that we call the \textbf{attacking graph} of $\mu$. Then two vertices of $\Gamma_{\mu}$ are adjacent if they are attacking one another. For this paper, \textbf{coloring} of a graph $G = (V,E)$ is a function $f: G \rightarrow \mathbb{N}$ such that $f(v)\neq f(w)$ whenever $\{v, w\} \in E.$ As a consequence, we have the following easy lemma:

\begin{lem}
For any weak composition $\mu$, the map $\sigma: \text{dg}'(\mu) \to [n]$ is a non-attacking filling if and only if $\widehat{\sigma}: \widehat{\text{dg}}(\mu) \to [n]$ defined by $\widehat{\sigma}(i,j) = \sigma(i,j)$ for $j \geq 1$ and $\widehat{\sigma}(i,j) = i$ otherwise is a coloring of $\Gamma_{\mu}$.
\end{lem}

\begin{proof}
This is immediate from the definition of a non-attacking filling. 
\end{proof}

These graphs turn out to relate to a fundamental family of graphs from the chromatic symmetric functions literature called Dyck graphs. Recall that a Dyck path is a lattice path from $(0,0)$ to $(n,n)$ that always stays above the line $y = x$. For example, consider the following:
    
    \[\begin{tikzpicture}
    \draw[red, thick, dashed] (0,0) -- (0,1)-- (0,2) -- (0,3) -- (1,3) -- (1,4) -- (2,4) -- (2,5) -- (3,5) -- (4,5) -- (5,5);
    \draw[blue, dotted, thick] (0,0) -- (5,5);
  \foreach \x in {0,1,2,3,4,5} {
    \foreach \y in {0,1,2,3,4,5} {
      \fill (\x, \y) circle (2pt); 
    }
  }
\end{tikzpicture}\]
From a Dyck path $\delta$, we can build a corresponding graph as follows. For a Dyck path from $(0,0)$ to $(n,n)$. The vertices of the \textbf{Dyck graph} $G_{\delta}$ are $V = [n]$, and $i$ and $j$ are adjacent for $1 \leq i < j \leq n$ if $(i-1,j)$ is on or below the Dyck path. Equivalently, label the boxes along the diagonal in order from $1$ to $n$. Then the boxes are the vertices and an earlier box is adjacent to a later box if the vertical line through the center of the earlier box intersects the horizontal line through the center of the later box below the Dyck path. For example, in the Dyck graph from our example, $1$ and $3$ are adjacent, while $1$ and $4$ are not adjacent. 
    \[\begin{tikzpicture}
    \draw[red, thick, dashed] (0,0) -- (0,1)-- (0,2) -- (0,3) -- (1,3) -- (1,4) -- (2,4) -- (2,5) -- (3,5) -- (4,5) -- (5,5);
    \draw (.5,.5) node {$1$};
    \draw (1.5,1.5) node {$2$};
    \draw (2.5,2.5) node {$3$};
    \draw (3.5,3.5) node {$4$};
    \draw (4.5,4.5) node {$5$};
    \draw[blue, dotted, thick] (.5,.7) -- (.5,2.5);
    \draw[blue, dotted, thick] (2.3,2.5) -- (.5,2.5);
  \foreach \x in {0,1,2,3,4,5} {
    \foreach \y in {0,1,2,3,4,5} {
      \fill (\x, \y) circle (2pt); 
    }
  }
\end{tikzpicture} \hspace{.5cm} \vline \hspace{.5cm} \begin{tikzpicture}
    \draw[red, thick, dashed] (0,0) -- (0,1)-- (0,2) -- (0,3) -- (1,3) -- (1,4) -- (2,4) -- (2,5) -- (3,5) -- (4,5) -- (5,5);
    \draw (.5,.5) node {$1$};
    \draw (1.5,1.5) node {$2$};
    \draw (2.5,2.5) node {$3$};
    \draw (3.5,3.5) node {$4$};
    \draw (4.5,4.5) node {$5$};
    \draw[blue, dotted, thick] (.5,.7) -- (.5,3.5);
    \draw[blue, dotted, thick] (3.3,3.5) -- (.5,3.5);
  \foreach \x in {0,1,2,3,4,5} {
    \foreach \y in {0,1,2,3,4,5} {
      \fill (\x, \y) circle (2pt); 
    }
  }
\end{tikzpicture}\]

Our goal is to show that attacking graphs are Dyck graphs. For this, we first need to construct a bijection between weak compositions and a special subset of Dyck paths. We define a map called the \textbf{Dyckification} $\mathcal{F}(\mu)$ of a weak composition $\mu$. 

Let $\mu = (\mu_{1}, \mu_{2},\dots,\mu_{n})$. In order to encode $\mathcal{F}(\mu)$, we use the description of a Dyck path as a sequence of $2|\mu|$ north steps and east steps, where $|\mu| := \sum_{i=1}^{n} \mu_{i}$. For this we define a filling of the augmented diagram $\widehat{dg}(\mu)$ with either $E$ or $EN$. An entry has label $E$ if there is no block above it or $EN$ if there is a block above it. An example filling would be as follows:
\[\begin{ytableau}
E & E & \none \\
EN & EN & E \\
\end{ytableau}\]
Then the Dyck path is the first $n$ North steps followed by the entries of the diagram read in the same order as the vertices. For the example, the resulting Dyck path is the example Dyck path above given by NNNENENEEE. Then $\mathcal{F}(\mu)$ is the Dyckification of $\mu$. 

\begin{lem}
The Dyckification map is a bijection between weak compositions and Dyck paths with no two consecutive north steps after the first east step.
\end{lem}

\begin{proof}
First, we verify that $\mathcal{F}(\mu)$ is a Dyck path. The first condition is that there must be $|\mu|+n$ north steps and $|\mu|+n$ east steps. Note that every block in the diagram has entry $E$ or $EN$, and a block has entry $E$ if and only if it has no block above it, by definition. Hence, each column has precisely one entry with label $E$ and all others $EN$. Thus, in the labeling, there are $|\mu|$ entries with label $EN$ and $n$ with label $E$. The first $n$ steps of the path before using the steps from the labeling all have label $N$. Hence, the path has $|\mu|+n$ east steps and $|\mu|+n$ north steps. 

The second condition is that: For all $j \leq 2(|\mu|+n)$, the first $j$ elements in the words have more north steps than east steps. Suppose this is not the case. Then there exists a minimal $1\leq j \leq 2(|\mu|+n)$ such that the number of north steps in the first $j$ elements is strictly less than the number of east steps in the first $j$ elements. The word starts with $n$ north steps, so $j > n$.

Suppose for the sake of contradiction that $\mathcal{F}(\mu)_{j} = N$. Then the first $j-1$ steps have the same amount of east steps but $1$ fewer north step. Hence, then the first $j-1$ steps would also have more east steps than north steps, a contradiction to the minimality of $j$. Thus, $\mathcal{F}(\mu)_{j} = E$.  
Note also that $j < 2(|\mu|+n)$, since the total sequence has the same number of north steps as east steps. Then prior to $j$ being added, there have been $a$ blocks with $EN$ steps, $b$ blocks with $E$ steps, and $n$ $N$ blocks. Since the total number of $E$ blocks is $n$, and the last block has label $E$, $b \leq n-1$. Hence, the total number of easts in the first $j$ steps is $a + b +1 \leq a + n$, the number of norths in the first $j$ steps, a contradiction to our assumption about $j$. Therefore, $\mathcal{F}(\mu)$ is actually a Dyck path. By construction, after the first $n$ north steps, each step is an E step or an EN step. Either way, there cannot be two north steps in a row. 

We define an inverse map algorithmically. Let $\delta$ be a Dyck path with no two norths in a row after the first $n$ norths. Then after the first $n$ norths the remaining word has no two consecutive norths. The subword following those n norths does not have consecutive norths and starts with an east. Hence, any north must first be preceded by an east and so the word may be written as a string of $E$ and $EN$ steps. 

Initialize $S = [n]$ and $g(\mu) = (0,0,\dots,0) \in \mathbb{Z}^{n}$. While $S \neq \emptyset$, do the following. Consider the next $|S|$ steps of the path: $X_{1}, X_{2},\dots, X_{|S|}$ with $X_{i} \in \{E, EN\}$ for all $i \in [|S|]$. Update $S = \{j \in S: X_{i_{j}} = EN\}$, where $i_{j}$ is defined such that $j$ is the $i_{j}$th largest element of $S$ and add $1$ to $g(\mu)_{i}$ for each $i \in S$. 

To be a Dyck path, each north step must be followed eventually by a corresponding to E step. Hence, as there are $n$ initial north steps, there must be at least $n$ $E$ steps in order to be a Dyck path. A set $S$ of those steps will be of the form $EN$, and so there will be $|S|$ many more $E$ steps after the first $n$ steps. This continues until $S = \emptyset$. Hence, this algorithm does indeed yield a weak composition in finite time.

It remains to prove that $f$ and $g$ are inverses of one another. Let $\mu$ be a weak composition. Then $\mathcal{F}(\mu)$ is a Dyck path with $n$ initial norths. In the run of the algorithm, $S$ is initialized as the basement and updated at each step to become the next row by construction. At each step, the algorithm updates the height of the corresponding columns by $1$ for each time $S$ changes. Hence, in the end, we have $g \circ \mathcal{F}(\mu) = \mu$. A similar argument works for the other direction.
\end{proof}

This bijection induces an isomorphism between the attacking graph and the corresponding Dyck graph. 

\begin{theorem}
\label{thm:Dyckification}
Let $\mu$ be a weak composition. Then $\Gamma_{\mu}$ is isomorphic to $G_{\mathcal{F}(\mu)}$, where $\mathcal{F}$ is the Dyckification map. In particular, every attacking graph is isomorphic to a Dyck graph. 
\end{theorem}

\begin{proof}
By construction, $\mathcal{F}(\mu)$ is a Dyck path from $(0,0)$ to $(|\mu| + n,|\mu|+n)$, so $\Gamma_{\mu}$ and $G_{\mathcal{F}(\mu)}$ have the same number of vertices. For the bijection, we can just use the identity map $\iota: [|\mu|+n] \to [|\mu|+n]$ by $\iota(i) = i$ if we use the labeling of the vertices of $\Gamma_{\mu}$ from bottom-left to top right reading row by row. 
For example, the indexing looks like the following for $\widehat{\text{dg}}(1,1,0)$:
\[\begin{ytableau}
4 & 5 & \none \\
1 & 2 & 3 \\
\end{ytableau}\]
Then all we need to argue is that $\{i,j\}$ is an edge in the attacking graph of $\mu$ if and only if it is an edge in the Dyck graph $\mathcal{F}(\mu)$. 

In the Dyck graph, we defined adjacency between $i$ and $j$ for $i < j$ by $i$ and $j$ are adjacent if $(i-1,j)$ is below the Dyck graph. An equivalent way to detect this property is to look for the highest point at step $i-1$ in the Dyck graph. That is the maximum number of north steps taken after $i-1$ east steps. Then we can say $i$ and $j$ are adjacent if the maximum number of north steps taken for $i-1$ east steps is at least $j$. 

We can construct a labeling $g_{\mu}: [|\mu| + n] \to [|\mu| + n]$ of the diagram by labeling each entry with the height of the Dyck path before following the first $E$ step from that entry. For $\mu = (1,1,0)$, we find the following labeling:
\[\begin{ytableau}
5 & 5 & \none \\
3 & 4 & 5 \\
\end{ytableau}\]
Before the first entry, there were $3$ steps north taken. That is why the first coordinate is $3$. In the second entry the coordinate is $4$, because there was an east north step. At the beginning of each entry of the diagram is an east step, so there are precisely $i-1$ east steps taken before reaching the $i$th entry. Hence, the number of north steps taken, the height of the entry, is precisely the maximum height at $i-1$. By the definition of adjacency for a Dyck graph, we have for all $j > i$ that $j$ is adjacent to $i$ in the Dyck graph for $\mathcal{F}(\mu)$ if and only if $j \leq g_{\mu}(i)$.

We can define another labeling by $h_{\mu}: [|\mu| + n] \to [|\mu|+ n]$ by $h_{\mu}(i)$ is the largest entry in the next row above and to the left of $i$, or if no such entry exists, the maximum of the row containing $i$. By definition of the attacking graph, $i$ and $j$ are adjacent for $i < j$ if and only if $j \leq h_{\mu}(i)$. Therefore, to prove these graphs are isomorphic, it suffices to show that $g_{\mu} = h_{\mu}$.

By definition $g_{\mu}(1) = n$, since there are $n$ north steps prior to the Dyck path beginning. Furthermore, since for the bottom left entry there is nothing above and strictly to the left of it, $h_{\mu}(1)$ is the maximum of the first row, which is $n$. Hence, $g_{\mu}(1) = h_{\mu}(1)$. Then, for the sake of induction, assume that $g_{\mu}(i-1) = h_{\mu}(i-1)$. Note that, by definition,
\[g_{\mu}(i) = \begin{cases} g_{\mu}(i-1) &\text{ if  the }i\text{th entry is E}  \\
g_{\mu}(i-1) + 1 &\text{ Otherwise.}\end{cases}.\]

It suffices to show the same recurrence holds for $h_{\mu}$. We resolve this by splitting into cases. 

Suppose that the $i-1$st entry is $E$. Then there is nothing above the entry at $i-1$. Then if $i-1$ and $i$ are in the same row, there is an entry above and strictly to the left of $i$ if and only if there is an entry above and strictly to the left of $i-1$. Furthermore, the rightmost one of these must be the same, since there is nothing above the $i-1$st entry. Hence, in that case, $g_{\mu}(i) = g_{\mu}(i-1) = h_{\mu}(i-1) = h_{\mu}(i)$. 

Suppose instead that $i-1$ and $i$ are in different rows, and there is nothing above $i-1$, then the $i-1$st entry must be in the rightmost column of its row, which is the row directly below the one containing the $i$th entry. Since $i-1$ has nothing above it, the rightmost entry of the row above it must be strictly to the left. Hence, $h_{\mu}(i-1)$ will be the maximum of that row. Since $i-1$ is in a distinct row from $i$, $i$ is the first entry of its row. It must therefore have no neighbors above and strictly to its left. Thus, $h_{\mu}(i)$ is also the maximum of that row. It follows then that $g_{\mu}(i) = g_{\mu}(i-1) = h_{\mu}(i-1) = h_{\mu}(i)$ as desired.

Suppose instead that $i-1$st entry is $EN$. Equivalently there is an entry directly above $i-1$. Then there are two cases. Suppose that $i-1$ and $i$ are in the same row. Then the entry above $i-1$ is strictly above and to the left of $i$ in the row above it. It is also maximal among all such elements and hence that entry equals $h_{\mu}(i)$. If there is some thing directly above and to the left of $i-1$, then $h_{\mu}(i-1)$ is in the same row as $h_{\mu}(i)$. Furthermore, by maximality, it will be equal to $h_{\mu}(i) - 1$. Therefore, $h_{\mu}(i) = h_{\mu}(i-1) + 1 = g_{\mu}(i-1) + 1 = g_{\mu}(i)$ in that case. 

The only remaining case is when $i-1$ and $i$ are in distinct rows. Then, by the same reasoning as in the previous case with two rows, $h_{\mu}(i)$ is the maximum of its row. Since $i$ is in a different row of $i-1$, $i-1$ must be the maximum of its row meaning that the element above it must also be the maximum of its row. Hence, $h_{\mu}(i)$ must be directly above $i-1$. Then $h_{\mu}(i)-1$ is either strictly to its left in the same row or no such element exists and so it is equal to $i-1$. In either case, $h_{\mu}(i-1) = h_{\mu}(i)-1$, so $h_{\mu}(i) = h_{\mu}(i-1) + 1 = g_{\mu}(i-1) + 1 = g_{\mu}(i)$. Therefore, in all cases $h_{\mu}(i) = g_{\mu}(i)$ meaning that $h_{\mu}(i) = g_{\mu}(i)$ for all $i \leq n$. Thus, $g_{\mu} = h_{\mu}$ meaning that the graphs are isomorphic. 

\end{proof}

In general, Dyck paths are in bijection with functions $f: [n] \to [n]$ such that $f(i+1) \geq \max(i+1, f(i)\}$ for all $i \in [n-1]$. These are called \textbf{Hessenberg functions} due to their role in the theory of Hessenberg varieties. We call the function $h_{\mu}$ from the proof of Theorem \ref{thm:Dyckification} the Hessenberg function of the weak composition $\mu$. As an immediate consequence of the proof of Theorem \ref{thm:Dyckification}, we have the following description of the graph from its Hessenberg function:

\begin{cor}
\label{cor:Hessenberg}
If $\mu$ is a weak composition with length $n$, then $G_{\mathcal{F}(\mu)}$ has vertices $[|\mu| + n]$, and $i, j \in [|\mu| + n]$ for $i \leq j$ are adjacent if and only if $j \leq h_{\mu}(i)$.
\end{cor}

Finally, we can connect back to non-symmetric analogs of chromatic polynomials. Recall that the \textbf{chromatic symmetric polynomial} for a graph $G$ is given by
\[\sum_{f \in C_{n}(G)} x^{f},\]
where $x^f:= \prod_{v \in V} x_{f(v)}$ and $C_{n}(G)$ is the set of $n$-colorings of $G$, i.e., those colorings of $G$ of the form $f:G \rightarrow \{1,\cdots ,n\}$. This polynomial is always symmetric, since one can permute colors and still yield a valid $n$-coloring of a graph. However, if one were to pre-color a set of the vertices, it would no longer be symmetric. In particular, we may define $C_{n}(G, c)$ to be the set of colorings of $G$ extending some coloring $c: V' \to [n]$ of an induced subgraph of $G$. In this case, associate to each coloring $f \in C_{n}(G, c)$ the monomial weight $x^f:= \prod_{v\in V'} x_{f(v)}.$ If the chromatic number is equal to the clique number as is the case for Dyck graphs, then a reasonable method to asymmetrize is to force a coloring of a maximal clique. This is precisely what happens here.

\begin{cor}
\label{cor:ChromInterpretation}
For any weak composition $\mu = (\mu_{1}, \dots, \mu_{n})$, the non-symmetric Macdonald polynomial $E_{\mu}$ has the same support as a non-symmetric chromatic polynomial of the Dyck graph $G_{\mathcal{F}(\mu)}$
\[\sum_{f \in C_{n}(G,c)}x^f.\]
for the coloring $c: \{(0,i): 1 \leq i \leq n\} \to [n]$ defined by $c(0,i) = i$.
\end{cor}

\section{Symmetries from Dyck Graph Colorings}
\label{sec:symmetries}
In order to prove that the non-symmetric Macdonald polynomials have $M$-convex supports, we apply the graph coloring interpretation.

\begin{definition}
    For $\mu \in \mathbb{Z}_{\geq 0}^{n}$ define $\cL(\mu)$ to be the set of all non-attacking labelings of $\mu$, i.e., labelings $\sigma$ of the boxes of $\dg'(\mu)$ such that the augmented filling $\widehat{\sigma}$ is non-attacking. Given any permutation $\gamma \in \mathfrak{S}_n$ and any (possibly attacking) diagram labeling $\sigma:\mu \rightarrow \{1, \dots, n\}$, we denote by $\gamma(\sigma)$ the diagram labeling given by $\gamma(\sigma)(\square):= \gamma( \sigma(\square)).$ For $\mu \in \mathbb{Z}_{\geq 0}^n $, define the multiplicity map $\mult: \cL(\mu) \rightarrow \mathbb{Z}_{\geq 0}^{n}$ as 
     $\mult(\sigma) = (\mult_{1}(\sigma),\dots,\mult_n(\sigma))$ where $\mult_j(\sigma) := \# ( \square \in \dg'(\mu) \mid \sigma(\square) = j ).$
\end{definition}

From Definition \ref{def: nsym macd} we see:
\begin{lem}\label{lem: labels mult}
    For all $\mu \in \mathbb{Z}_{\geq 0}^n$, $\mult(\cL(\mu)) = \supp(E_{\mu}).$
\end{lem}

For the remainder of this section, fix some weak composition $\mu \in \mathbb{Z}_{\geq 0}^n$, $a \geq 1$, and $m \geq 2$. Let $N:= m+n$. We are interested in understanding the relationship between the sets of non-attacking labelings corresponding to the weak compositions $\nu_1,\dots,\nu_m$ defined by  $\nu_i:= 0^{i-1}*a*0^{m-i}*\mu$ for $1\leq i \leq m$, as the left-most nonzero column corresponding to $a \geq 1$ is shifted from left to right. Here the $*$'s denote concatenations of weak compositions. We may identify the boxes of each of the diagrams $\widehat{\dg(\nu_1)},\dots,\widehat{\dg(\nu_m)}$ in the obvious way: identify the basement boxes, identify the boxes corresponding to $\mu$ which do not shift, and identify those boxes in the moving column corresponding to $a.$ As such, any (possibly attacking) labeling of some $\nu_i$ may be thought of as a (possibly attacking) labeling of any $\nu_{j}$ for all $1\leq i,j\leq m.$ 
 
\begin{example}
    For instance, when $\mu = (2,1),$ $a =2,$ and $m = 3$, we are considering the weak compositions $\nu_1 = (2,0,0,2,1), \nu_2 = (0,2,0,2,1)$, and $\nu_3 = (0,0,2,2,1)$ with corresponding diagrams: 

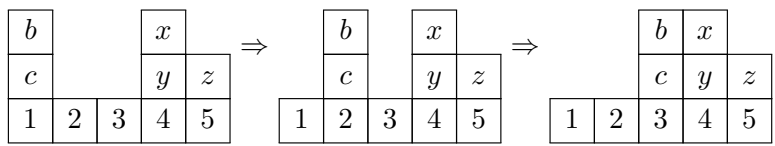
\begin{center}
        \begin{ytableau}
       b & \none & \none & x  & \none \\
       c & \none & \none & y &  z \\
       1 & 2 & 3 & 4 & 5 \\ 
       \end{ytableau} 
       ~~$\Rightarrow  $~~
       \begin{ytableau}
       \none & b & \none & x  & \none \\
       \none & c & \none & y &  z \\
       1 & 2 & 3 & 4 & 5 \\ 
       \end{ytableau} 
       ~~$\Rightarrow $~~
       \begin{ytableau}
       \none & \none & b & x  & \none \\
       \none & \none & c & y &  z \\
       1 & 2 & 3 & 4 & 5 \\ 
       \end{ytableau}
       
\end{center}
\end{example}

For $1\leq i< j \leq n$, we write $s_{i,j}$ for the transposition swapping the $i$ and $j$ entries of vectors in $\mathbb{Z}^n$.
Use $\sqcup$ to denote disjoint unions.

In order to understand the impact of this operation of moving from $\nu_{i}$ to $\nu_{i+1}$ on the monomial, we may first understand how it impacts the corresponding Dyck graph. A standard graph operation is to remove an edge. Given a graph $G = (V,E)$ and an edge $e \in E$, we let $G \setminus e = (V, E \setminus e)$. 

\begin{lem}
\label{lem:GraphSymmetrizing}
Let $h_{\nu_{i}}$ and $h_{\nu_{i+1}}$ be the Hessenberg functions of the Dyck graphs of $\nu_{i}$ and $\nu_{i+1}$. Then 
\[h_{\nu_{i+1}}(j) = \begin{cases} h_{\nu_{i}}(j)-1 &\text{ if } j = i+1 \\ 
h_{\nu_{i}}(j)  &\text{Otherwise}\end{cases} \]
In particular, $G_{f(\nu_{i+1})} = G_{f(\nu_{i})} \setminus (i+1,n+1)$, where $n$ is the length of the basement of $\nu_{i}$.
\end{lem}

\begin{proof}
Tracing the proof of Theorem \ref{thm:Dyckification}, moving a column to the right does not change the reading order or the number of boxes, so the labeling of the vertices is the same. The filling of boxes by choices of $E$ or $EN$ is determined solely by which boxes having a box lying directly above it. When moving the first column from $i$ to $i+1$, the only change is that box $i$ will no longer have a box directly above it and box $i+1$ will now have a box directly above it. That is, the labeling will change from box $i$ having $EN$ to $E$ and box $i+1$ will have an $EN$ instead of an $E$. In particular, this means that for all $j \neq i+1$, the number of $EN$ labeled entries strictly before $j$ stays the same, while for $j = i+1$, it goes down by $1$. Then, by the definition of the Hessenberg function from Theorem \ref{thm:Dyckification}, \[h_{\nu_{i+1}}(j) = \begin{cases} h_{\nu_{i}}(j)-1 &\text{ if } j = i+1 \\ 
h_{\nu_{i}}(j)  &\text{Otherwise}\end{cases}. \]
By Corollary \ref{cor:Hessenberg}, the graph $G_{f(\nu_{i+1})}$ is defined by $a$ and $b$ being adjacent for $a < b$ if and only if $b \leq h_{\nu_{i+1}}(a)$. Thus, there is precisely one edge removed given by $(i+1, h_{\nu_{i}}(i+1))$. By our definition of $\nu_{i}$, there was only one tower appearing before the $i+1$st basement entry, and so there was only one $EN$ step prior to $i+1$. Hence, $h_{\nu_{i}}(i+1) = n+1$ meaning that  $G_{f(\nu_{i+1})} = G_{f(\nu_{i})} \setminus (i+1,n+1)$. 
\end{proof}

Thus moving from $\nu_{i}$ to $\nu_{i+1}$ on the level of the graph just removes a single edge to the basement. Note that Dyck graphs enjoy several rich properties as they are equivalent to unit interval graphs. 

For us, the only property we need to leverage here is that these graphs are claw-free. Namely, they do not contain an induced subgraph isomorphic to $G = (V,E)$ with $V = \{1,2,3,4\}$ and $E = \{\{1,2\}, \{1,3\},\{1,4\}\}$ called the claw. Stanley considered claw-freeness when studying Schur positivity of chromatic symmetric polynomials in \cite{StanleyNiceGraphs}. In particular, in the discussion following Proposition $1.5$ of that paper, he observes that if a symmetric polynomial is Schur positive, then its support must be an order ideal with respect to dominance ordering. He showed in Proposition $1.6$ that the support of a chromatic symmetric polynomial is an order ideal in dominance ordering if the graph is claw-free. This observation was the key tool for the proof of saturation for chromatic symmetric polynomials of Dyck graphs in \cite{ChromaticSNP}. 

Breaking of symmetry for our chromatic interpretation stops us short of being able to apply this dominance ordering observation directly. However, Stanley's argument still buys us a lot. The key observation is that a claw-free bipartite graph must be a disjoint union of paths and cycles. This is because if a vertex had degree at least $3$, then the induced graph on that vertex and its neighborhood must contain a claw. Hence, every vertex would be of degree at most $2$, which implies that every connected component of the graph is a path or a cycle. Since claw-free graphs are closed under induced subgraphs, given a fixed coloring the induced subgraph on any pair of colors is bipartite by construction and so must be a disjoint union of paths and cycles. Then to obtain a new coloring, one can swap the colors on a path or a cycle. We apply this observation here. In what follows, recall that $s_{i}: \mathbb{R}^{n} \to \mathbb{R}^{n}$ denotes the linear map swapping coordinates $i$ and $i+1$. As a standard abuse of notation, for each $i \in [n-1]$, we also let $s_{i}: [n] \to [n]$ defined by $s_{i}(i) = i+1, s_{i}(i+1) = i,$ and $s_{i}(k) = k$ for all $k \notin \{i,i+1\}$, where the meaning is made clear by the input.  

\begin{lem}
\label{lem:swapping}
For $\nu_{i}$ and $\nu_{i+1}$, we have
\begin{itemize}
    \item $\text{mult}(\mathcal{L}(\nu_{i+1})) = \text{mult}(\mathcal{L}(\nu_{i})) \cup s_{i}\text{mult}(\mathcal{L}(\nu_{i+1}))$
    \item $\text{mult}(\mathcal{L}(\nu_{i+1})) \subseteq \text{mult}(\mathcal{L}(\nu_{i})) \cup (\text{mult}(\mathcal{L}(\nu_{i})) + e_{i+1}-e_{i})$.
\end{itemize}
\end{lem}

\begin{proof}
Note that by Corollary \ref{cor:ChromInterpretation}, $\mathcal{L}(\nu_{i+1})$ is the set of $n$-colorings of $G_{f(\nu_{i+1})}$ such that the color of vertex $j$ is $j$ for $j \leq n$. Note that $G_{f(\nu_{i})}$ and $G_{f(\nu_{i+1})}$ have the same vertex set $[|\nu_{i}|+n]$. Let $\rho: [|\nu_{i}|+n] \to [n]$ be a coloring such that $\rho(j) = j$ for all $j \leq n$. Then, since $G_{f(\nu_{i+1})}$ is a subgraph of $G_{f(\nu_{i})}$ by Lemma \ref{lem:GraphSymmetrizing}, if $\rho$ is a coloring of $G_{f(\nu_{i})}$, it is also a coloring of $G_{f(\nu_{i+1})}$. Hence, $\mathcal{L}(\nu_{i}) \subseteq \mathcal{L}(\nu_{i+1})$. For example, the following are both non-attacking fillings:
\[
 \begin{ytableau}
       4 & \none & \none & 1  & \none \\
       1 & \none & \none & 2 &  5 \\
       1 & 2 & 3 & 4 & 5 \\ 
       \end{ytableau} 
       ~~ \Rightarrow  ~~
       \begin{ytableau}
       \none & 4 & \none & 1  & \none \\
       \none & 1 & \none & 2 &  5 \\
       1 & 2 & 3 & 4 & 5 \\ 
       \end{ytableau} \]
In fact, the graph of $G_{f(\nu_{i})}$ only has one additional edge between $n+1$ and the basement element $i+1$. The only colorings of $G_{f(\nu_{i+1})}$ that are not colorings of $G_{f(\nu_{i})}$ are those for which entry $n+1$ and entry $i+1$ have distinct colors. By assumption $i+1 \leq n$, so $\rho$ is a coloring of $G_{f(\nu_{i+1})}$ that is not a coloring of $G_{f(\nu_{i})}$ if and only if entry $n+1$ has color $i+1$. For example, the right is a non-attacking filling, but the left is not:
\[
 \begin{ytableau}
       4 & \none & \none & 1  & \none \\
       \textcolor{red}{\textbf{2}} & \none & \none & 1 &  5 \\
       1 & \textcolor{red}{\textbf{2}} & 3 & 4 & 5 \\ 
       \end{ytableau} 
       ~~ \Leftarrow  ~~
       \begin{ytableau}
       \none & 4 & \none & 1  & \none \\
       \none & 2 & \none & 1 &  5 \\
       1 & 2 & 3 & 4 & 5 \\ 
       \end{ytableau} \]
This stronger observation is what gives rise to the symmetries. First, note that, since permuting colors preserves being a coloring, $s_{i} \circ \rho$ is still a coloring. However, $s_{i} \circ \rho(i) = i+1$ and $s_{i} \circ \rho(i+1) = i$. Define $\widehat{\rho}$ by 
\[\widehat{\rho}(j) = \begin{cases}
    j &\text{if } j \in \{i, i+1\} \\
    s_{i} \circ \rho(j) &\text{Otherwise.}
\end{cases} \]
We claim that if $\rho$ is a coloring of $G_{f(\nu_{i})}$, then $\widehat{\rho}$ is a coloring of $G_{f(\nu_{i+1})}$. Then, since $s_{i} \circ \rho$ is a coloring, the only way for $\widehat{\rho}$ to not be a coloring is for $i$ or $i+1$ to have neighbors of the same color. Note that $h_{\nu_{i+1}}(i) = h_{\nu_{i+1}}(i+1) = n$ by Lemma \ref{lem:GraphSymmetrizing}, so all neighbors of $i$ and $i+1$ have coordinate at most $n$ and thus by construction have distinct colors. Therefore, $\widehat{\rho}$ is still a valid coloring. For example, for $\rho$ on the left, $\widehat{\rho}$ is the coloring on the right.

\[
 \begin{ytableau}
       4 & \none & \none & 1  & \none \\
       1 & \none & \none & 2 &  5 \\
       1 & 2 & 3 & 4 & 5 \\ 
       \end{ytableau} 
       ~~ \Rightarrow  ~~
       \begin{ytableau}
       \none & 4 & \none & \textcolor{blue}{\mathbf{2}}  & \none \\
       \none & \textcolor{blue}{\mathbf{2}} & \none & \textcolor{blue}{\mathbf{1}} &  5 \\
       1 & 2 & 3 & 4 & 5 \\ 
       \end{ytableau} \]

Note that $\widehat{\rho}$ swaps the number of vertices colored with color $i$ and color $i+1$. It follows then that \[\text{mult}(\mathcal{L}(\nu_{i+1})) \supseteq \text{mult}(\mathcal{L}(\nu_{i})) \cup s_{i} \text{mult}(\mathcal{L}(\nu_{i+1})).\]

To complete the proof, it suffices to show that $\text{mult}(\mathcal{L}(\nu_{i+1})) \setminus \text{mult}(\mathcal{L}(\nu_{i})) \subseteq s_{i}\text{mult}(\mathcal{L}(\nu_{i}))$ and $\text{mult}(\mathcal{L}(\nu_{i+1})) \setminus \text{mult}(\mathcal{L}(\nu_{i})) \subseteq \text{mult}(\mathcal{L}(\nu_{i})) + (e_{i+1}-e_{i})$.

Let $\gamma$ be a coloring in $\mathcal{L}(\nu_{i+1}) \setminus \mathcal{L}(\nu_{i})$. Then, from what we have already shown, $\gamma(n+1) = i+1$. There are two ways to construct a coloring in $\mathcal{L}(\nu_{i})$ from $\gamma$.

First, construct $\widehat{\gamma}$ as we did for $\widehat{\rho}$. By the same argument, this remains a coloring. The number of vertices colored by $i$ and $i+1$ swaps, so $\text{mult}(\mathcal{L}(\nu_{i+1})) \setminus \text{mult}(\mathcal{L}(\nu_{i})) \subseteq s_{i}\text{mult}(\mathcal{L}(\nu_{i}))$. For example, for $\gamma$, on the right, we obtain $\widehat{\gamma}$ on the left:
For example, if $\rho$ were the coloring on the right, $\widehat{\rho}$ yields the coloring on the left:
\[
 \begin{ytableau}
       4 & \none & \none & \textcolor{blue}{\textbf{2}}  & \none \\
       \textcolor{blue}{\textbf{1}} & \none & \none & \textcolor{blue}{\textbf{2}} &  5 \\
       1 & 2 & 3 & 4 & 5 \\ 
       \end{ytableau} 
       ~~ \Leftarrow  ~~
       \begin{ytableau}
       \none & 4 & \none & 1  & \none \\
       \none & 2 & \none & 1 &  5 \\
       1 & 2 & 3 & 4 & 5 \\ 
       \end{ytableau} \]

Alternatively, note that the graph is claw-free, so the induced subgraph of $G_{\nu_{i+1}}$ on the pair of colors $i$ and $i+1$ is a disjoint union of paths and cycles. Consider the connected component containing entry $n+1$. Note that, by Lemma \ref{lem:GraphSymmetrizing}, as already used previously, $n+1$ is not adjacent to $i$ or $i+1$. Thus, swapping colors on the path or cycle containing $n+1$ yields another valid coloring of $G_{\nu_{i+1}}$. In the same example of $\gamma$, it yields the coloring:
\[
 \begin{ytableau}
       4 & \none & \none & 1  & \none \\
       \textcolor{blue}{\textbf{1}} & \none & \none & \textcolor{blue}{\textbf{2}} &  5 \\
       1 & 2 & 3 & 4 & 5 \\ 
       \end{ytableau} 
       ~~ \Leftarrow  ~~
       \begin{ytableau}
       \none & 4 & \none & 1  & \none \\
       \none & \textcolor{red}{\textbf{2}} & \none & \textcolor{red}{\textbf{1}} &  5 \\
       1 & 2 & 3 & 4 & 5 \\ 
       \end{ytableau} \] 

Note that, since $n+1$ corresponds to the first box in its row, all of its neighbors must be either in the basement or in the previous row. Hence, the component must be a path starting at $n+1$. Thus, swapping the colors either decreases the number of $i+1$ entries by $1$ and increases the number of $i$ entries by $1$ or remains the same. It follows that the multiplicity is contained in $\text{mult}(\mathcal{L}(\nu_{i})) \cup (\text{mult}(\mathcal{L}(\nu_{i})) + e_{i+1} - e_{i})$.
\end{proof}

Putting together Lemma \ref{lem: labels mult} and Lemma \ref{lem:swapping} yields the following, which is the key insight in order to ensure that M-convexity is preserved.

\begin{cor}\label{cor: supp set reflection condition}
    For all $1\leq i \leq m-1$, as subsets of $\mathbb{Z}^n,$
    $$ s_{i}(\supp(E_{\nu_i})) \subseteq \supp(E_{\nu_i}) \cup ( \supp(E_{\nu_i}) + e_{i+1}-e_{i}).$$
\end{cor}

\section{M-convexity}
\label{sec:MConvexity}
\subsection{A Geometric Lemma}

In this section, we prove a technical geometric lemma about preserving M-convexity required for the proof of Theorem \ref{thm:MacdonaldSNP}. To start, M-convexity has a standard definition in terms of an exchange axiom coming from discrete convex analysis. We rely on the following equivalent definition given in Theorem $1.9$ of \cite{Murota}:

\begin{definition}
    Let $P$ be a polytope. Then $P$ is a \textbf{generalized permutahedron} \cite{GenPermOrig} if for all pairs of adjacent vertices $\mathbf{u}, \mathbf{v}$ of $P$, there exists $i, j \in [n]$ such that $\mathbf{u} - \mathbf{v} \in \text{span}(e_{i} -e_{j})$. Let $S \subseteq \mathbb{Z}^{n}$. Then $S$ is \textbf{M-convex} if $S = P \cap \mathbb{Z}^{n}$ for some generalized permutahedron $P$ with vertices in $\mathbb{Z}^{n}$.
\end{definition}

Recall that $s_{i,j}$ denotes the reflection with respect to the linear hyperplane orthogonal to $e_{i} - e_{j}$ or equivalently swapping the $i$th and $j$th coordinates with $s_{i} = s_{i,i+1}$. Notice that the $M$-convexity property is invariant under $\mathbb{Z}^n$ lattice translations and simple root direction reflections $s_{i}$. Using Proposition \ref{prop:Knop-Sahi}, we obtain the following:

\begin{cor}
\label{cor: pi_support}
     For all $\mu \in \mathbb{Z}^n$, $\supp(E_{\pi(\mu)}) = e_1 + s_{1,2}\cdots s_{n-1,n}(\supp(E_{\mu}))$ and thus $\Newt(E_{\pi(\mu)}) = e_1 + s_{1,2}\dots s_{n-1,n}(\Newt(E_{\mu})).$ In particular, if $\supp(E_{\mu})$ is $M$-convex, then so is $\supp(E_{\pi(\mu)}).$
\end{cor}

The next lemma is motivated by Corollary \ref{cor: supp set reflection condition} and is essential to the proof of the main theorem. 

\begin{lem}
\label{lem:StrongMConvexity}
Let $L \subseteq \mathbb{Z}^{n}$ be an $M$-convex set and let $i,j \in \{1,\dots, n\}$ be distinct. If $s_{i,j}L \subseteq L \cup (L+e_{i} -e_{j})$, then $L \cup s_{i,j}L$ is $M$-convex. 
\end{lem}

\begin{proof}
First, we prove that $\conv(L \cup s_{i,j}L) \cap \mathbb{Z}^n = L \cup s_{i,j}L.$ Suppose not, so that there exists $x \in \conv(L \cup s_{i,j}L)\cap \mathbb{Z}^n \setminus (L \cup s_{i,j} L).$
Recall that, by Corollary 46.2c of \cite{SchrijverEncyclopedia}, a Minkowski sum of M-convex sets is M-convex. Hence, since $L$ is $M$-convex and $\{0, e_{i}-e_{j}\}$ is also $M$-convex,
\[L \cup (L + e_{i} - e_{j}) = L + \{0,e_{i}-e_{j}\}\] is also $M$-convex.

It follows that
\[L \cup (L+e_{i} -e_{j}) = \conv(L \cup (L+e_{i} -e_{j})) \cap \mathbb{Z}^{n} \supseteq \conv(L \cup s_{i,j}L) \cap \mathbb{Z}^{n}.\]
Hence, $x \in L \cup (L+ e_{i} - e_{j})$. By assumption, $x \notin L \cup s_{i,j} L$, so $x \in L+ e_{i} - e_{j}$ meaning that $x = \ell + e_{i} - e_{j}$ for some $\ell \in L$. Furthermore, note that 
\[s_{i,j} x \in s_{i,j} \conv(L \cup s_{i,j} L) \cap \mathbb{Z}^{n} = \conv(L \cup s_{i,j}L) \cap \mathbb{Z}^{n} \subseteq L \cup (L + e_{i} - e_{j}). \]
Since $x \notin L \cup s_{i,j} L$ by assumption, $s_{i,j} x \notin L$, meaning that $s_{i,j}x \in L  + e_{i} - e_{j}$. Hence, there exists $\ell' \in L$ such that 
\[s_{i,j} \ell - (e_{i} - e_{j})= s_{i,j}(\ell + e_{i} - e_{j}) = s_{i,j} x = \ell' + e_{i} - e_{j}.\]
Thus, 
\[s_{i,j} \ell + \ell' = (s_{i,j} x + e_{i}- e_{j}) + (s_{i,j} x - (e_{i} -e_{j})) = 2 s_{i,j}x.  \]
Thus, $s_{i,j} x \in \conv(s_{i,j} \ell, \ell')$, where $\ell' - s_{i,j} \ell$ is parallel to $e_{i} - e_{j}$. 

Two possibilities remain: either $s_{i,j} \ell \in L$ or $s_{i,j} \ell \notin L$. Suppose first that $s_{i,j} \ell \in L$. Then by $M$-convexity, $x \in \conv(s_{i,j} \ell, \ell') \cap \mathbb{Z}^{n} \subseteq \text{conv}(L) \cap \mathbb{Z}^{n} =  L$, a contradiction to our assumption that $x \notin L$. 

Suppose instead that $s_{i,j} \ell \notin L$. Then 
\[s_{i,j} \ell \in s_{i,j}L \setminus L \subseteq (L \cup (L + e_{i} - e_{j})) \setminus L \subseteq L + e_{i} - e_{j}.\]
It follows that $s_{i,j}x = s_{i,j}\ell - (e_{i} -e_{j}) \in L$, a contradiction again since $x \notin s_{i,j}L$.

Therefore, we reach a contradiction to our assumption that there existed $x \in (\conv(L \cup s_{i,j} L) \cap \mathbb{Z}^{n})\setminus (L\cup s_{i,j} L)$ meaning that 
\[L \cup s_{i,j} L = \conv(L \cup s_{i,j} L) \cap \mathbb{Z}^{n}.\]

Thus, to show $L \cup s_{i,j}L $ is M-convex, it suffices to show that $\text{conv}(L \cup s_{i,j} L)$ is a generalized permutahedron with integral vertices. Note that $L \subseteq \mathbb{Z}^{n}$, so $s_{i,j} L \subseteq \mathbb{Z}^{n}$. Thus, $\text{conv}(L \cup s_{i,j}L)$ must also have vertices in $L \cup s_{i,j} L \subseteq \mathbb{Z}^{n}$ meaning that it is integral. It suffices then to show it is a generalized permutahedron. Suppose it is not. Then it has two vertices $\mathbf{u}$ and $\mathbf{v}$ such that the edge from $\mathbf{u}$ to $\mathbf{v}$ is not parallel to $e_{i}-e_{j}$ for any $i,j \in [n]$. In particular, $\mathbf{u}, \mathbf{v} \in L \cup s_{i,j}L$.

Suppose that $\mathbf{u}, \mathbf{v} \in L$. Then, since they are adjacent in $\text{conv}(L \cup s_{i,j} L)$, they also must be adjacent in $L$. Hence, by M-convexity of $L$, the edge between $\mathbf{u}$ and $\mathbf{v}$ is parallel to $e_{i} -e_{j}$ for some $i,j \in [n]$, a contradiction.  Since $L$ is M-convex, $s_{i,j}L$ is also $M$-convex. It follows that supposing that $\mathbf{u}, \mathbf{v} \in s_{i,j} L$ leads to precisely the same contradiction.  If $\mathbf{v} = s_{i,j} \mathbf{u}$, then $\mathbf{v} - \mathbf{u}$ is parallel to the root $e_{i} -e_{j}$. It follows that $\mathbf{v} \neq s_{i,j}\mathbf{u}$, a contradiction once again. Hence, $\mathbf{u}, \mathbf{v}, s_{i,j} \mathbf{u},$ and $s_{i,j}\mathbf{v}$ must all be pairwise distinct. Furthermore, without loss of generality, $\mathbf{u} \in L$ and $\mathbf{v} \in s_{i,j}L \setminus L$. 

Note that, since $\mathbf{u}$ and $\mathbf{v}$ are adjacent vertices of $\text{conv}(L \cup s_{i,j}L)$, $s_{i,j} \mathbf{u}$ and $s_{i,j} \mathbf{v}$ are also adjacent vertices of $\text{conv}(L \cup s_{i,j}L)$. It follows that they cannot both be in $L$. Since $\mathbf{v} \in s_{i,j} L$, and $s_{i,j}$ is an involution, $s_{i,j} \mathbf{v} \in L$ meaning that $s_{i,j} \mathbf{u} \in s_{i,j}L \setminus L$. Hence, $s_{i,j} \mathbf{u} = \mathbf{u}' + e_{i} - e_{j}$ for some $\mathbf{u}' \in L$. 

Observe that, since $s_{i,j}$ is the reflection about the hyperplane orthogonal to $e_{i} - e_{j}$,
$\mathbf{u} - s_{i,j} \mathbf{u} = \alpha(e_{i} -e_{j})$ for some $\alpha \in \mathbb{R}$. Note that $\mathbf{u} \neq s_{i,j} \mathbf{u}$, so $\alpha \neq 0$. Suppose that $\alpha > 0$. Since $s_{i,j} \mathbf{u} = \mathbf{u}' + e_{i} - e_{j}$, and $s_{i,j} \mathbf{u} + \alpha(e_{i} -e_{j}) = \mathbf{u}$, by our assumption that $\alpha > 0$, we must have that $s_{i,j} \mathbf{u}$ lies on the interior of the line segment from $\mathbf{u}'$ to $\mathbf{u}$. However, $s_{i,j} \mathbf{u}$ is a vertex. Hence, $\alpha < 0$. 

By similar reasoning, $\mathbf{v} - s_{i,j} \mathbf{v} = \beta(e_{i} - e_{j})$. Again, $\mathbf{v} \neq s_{i,j} \mathbf{v}$, so $\beta \neq 0$. Since $\mathbf{v} = \mathbf{v}' + e_{i} -e_{j}$ and $s_{i,j} \mathbf{v} = \mathbf{v} - \beta(e_{i} - e_{j})$, if $\beta < 0$, then $\mathbf{v}$ lies on the segment between $\mathbf{v}'$ and $s_{i,j}\mathbf{v}$. Again, this contradicts our assumption $\mathbf{v}$ is a vertex. Hence, $\beta > 0$.

Then, $s_{i,j} \mathbf{u} = \mathbf{u} + |\alpha|(e_{i} -e_{j})$ and $s_{i,j} \mathbf{v} = \mathbf{v} - |\beta|(e_{i} -e_{j})$. It follows that
\begin{align*}
    \frac{|\beta| \mathbf{u} + |\alpha|\mathbf{v}}{|\alpha| + |\beta|} &= \frac{|\beta| \mathbf{u} + |\beta||\alpha|(e_{i}-e_{j}) + |\alpha|\mathbf{v}-|\beta||\alpha|(e_{i}-e_{j})}{|\alpha| + |\beta|}\\
    &= \frac{|\beta|(\mathbf{u} + |\alpha|(e_{i}-e_{j})) + |\alpha|(\mathbf{v}-|\beta|(e_{i}-e_{j}))}{|\alpha| + |\beta|} \\
    &= \frac{|\beta| s_{i,j} \mathbf{u} + |\alpha|s_{i,j}\mathbf{v}}{|\alpha| + |\beta|}.
\end{align*}
Geometrically, this means that the line segments from $\mathbf{u}$ to $\mathbf{v}$ and $s_{i,j} \mathbf{u}$ to $s_{i,j} \mathbf{v}$ intersect on their interiors. However, this contradicts our assumption that $\mathbf{u}$ and $\mathbf{v}$ form an edge. Therefore, $\text{conv}(L \cup s_{i,j}L)$ is a generalized permutahedron, and as a result, $L \cup s_{i,j}L$ is M-convex.

\end{proof}

\subsection{Proof of main theorem}

With all of this in place, we may prove the M-convexity of supports for non-symmetric Macdonald polynomials.

\begin{proof}[Proof of Theorem \ref{thm:MacdonaldSNP}]
Let $E_{\mu}$ be the non-symmetric Macdonald polynomial for the weak composition $\mu = (\mu_{1}, \dots, \mu_{n})$. If $\mu = (0,\dots,0)$, $\Newt(E_{(0,\dots,0)}) = \{ 0\}$ is trivially $M$-convex. Suppose for the sake of contradiction that there exists a weak composition $\mu^{\ast}$ such that $E_{\mu^{\ast}}$ is not $M$-convex. Let $k$ be minimal such that there exists a weak composition $\mu$ such that $|\mu| = k$, and $E_{\mu}$ does not have $M$-convex support. Furthermore, choose some $\mu$ such that $\ell = \min(\{i: \mu_{i} >0\})$ is minimal among all such $\mu$ with $|\mu| = k$ and does not have M-convex support. 

Suppose that $\ell =1$, so $\mu_{1} > 0$. Since $\mu_{1} >0$, $(\mu_{2}, \dots,\mu_{n}, \mu_{1}-1)$ is a valid weak composition. Furthermore, $|(\mu_{2}, \dots, \mu_{n}, \mu_{1}-1)| = k-1$, so $E_{(\mu_{2}, \dots, \mu_{n-1}, \mu_{1}-1)}$ has $M$-convex support. Then, by Corollary \ref{cor: pi_support}, $E_{\mu}$ must also have $M$-convex support, a contradiction. 

Thus, $\ell > 1$. Define $\mu'_{i} = \mu_{i}$ if $i > \ell$, $0$ if $i \neq \ell -1$, and $\mu_{\ell}$ if $i = \ell -1$. By the minimality of $\mu$, $E_{\mu'}$ has M-convex support. However, by Corollary \ref{cor: supp set reflection condition}, we may apply Lemma \ref{lem:StrongMConvexity} to $\supp(E_{\mu'})$ to see that $E_{\mu}$ must also have M-convex support, yielding a contradiction. Therefore, for all weak compositions $\mu$, $E_{\mu}$ must have $M$-convex support. Lastly, by the definition of M-convexity, this means that for all weak compositions $\mu$, the support of $E_{\mu}$ is the set of integer points in its convex hull. Therefore, $E_{\mu}$ is saturated. 
\end{proof}

Note that, based on the work in \cite{LogConcaveManyPoly}, one may naturally ask whether the non-symmetric Macdonald polynomials may be (denormalized) Lorentzian for every choice of $q$ and $t$ for $0 < q,t < 1$. Via direct computation in SageMath \cite{sage}, we checked and found that 
\begin{align*}
    E_{[0,1,3]}(x_{0},x_{1},x_{2}) &= (296/651)x_0^3x_1 + (88/217)x_0^2x_1^2 + (8/21)x_0 x_1^3 + (8/21)x_0^3x_2 + (36/31)x_0^2x_1x_2 \\
    &+ (740/651)x_0x_1^2x_2 + (4/7)x_1^3x_2 + (4/7)x_0^2x_2^2 + (986/651)x_0x_1x_2^2 \\
    &+ (6/7)x_1^2x_2^2 + (2/3)x_0x_2^3 + x_1x_2^3    
\end{align*}
 is neither Lorentzian nor denormalized Lorentzian for $q = t = 1/2$. Namely, applying $\partial_{x_{1}} \partial_{x_{2}}$ to $E_{[0,1,3]}(x_{1},x_{2},x_{3})$ or its normalization yields a quadratic form with Gram matrix that has more than one positive eigenvalues. Without normalizing, it is also not log-concave along root directions. Namely, the sequence $(296/651, 88/217, 8/21)$ of coefficients of $x_{0}^{3}x_{1}, x_{0}^{2}x_{1}^{2},$ and $x_{0}x_{1}^{3}$ is not log-concave. However, we did computations for several examples in Sage for $q=t=1/2$, and the normalizations do satisfy log-concavity of coefficients along root directions for all cases we tested. We leave deciding whether this is the case an open problem. It remains possible that these polynomials are Lorentzian or denormalized Lorentzian for some choices of $q$ and $t$. In particular, Demazure characters are conjecturally denormalized Lorentzian \cite{LogConcaveManyPoly}, and they are a limiting case of the non-symmetric Macdonald polynomials for $q, t \to 0$.

\subsection{Conjecture of Monical, Tokcan, and Yong}

In this section, we show that Theorem \ref{thm:MacdonaldSNP} resolves Conjecture 3.8 \cite{SNPOriginal} of Monical, Tokcan, and Yong in the affirmative. 

\begin{definition}\label{def: bruhat}
    Define the Bruhat order $\leq_{S}$ on $\mathbb{Z}^n$ as the transitive closure of the following relations. Given $\lambda \in \mathbb{Z}^n$ and $1\leq i< j \leq n$ with $\lambda_{i} < \lambda_{j}$, we set $\lambda > s_{ij}(\lambda)$, and if in addition $\lambda_{j} - \lambda_i > 1$, then $ s_{ij}(\lambda) >_{S} \lambda +e_i -e_j.$
\end{definition}
 
Let $\widehat{\mathcal{P}}_{\alpha} = \text{conv}(\{\beta: \beta \leq_{S} \alpha\})$. Monical, Tokcan, and Yong posed the following conjecture:

\begin{conj}[Conjecture 3.8 in \cite{SNPOriginal}]
\label{conj: snp conj}
If $\beta \in \widehat{\mathcal{P}_{\alpha}}$ and $\beta \in \mathbb{Z}^{n}_{\geq 0}$, then $\beta \leq_{S} \alpha$. 
\end{conj}

The Bruhat order relates to non-symmetric Macdonald polynomials via the following triangularity result of Sahi:

\begin{prop}\label{prop:bruhat interval}\cite{Sahi}
    For all $\mu \in \mathbb{Z}^n,$ the non-symmetric Macdonald polynomial $E_{\mu}$ has a triangular expansion of the form $E_{\mu} = \sum_{\lambda \leq \mu} a_{\lambda}(q,t) x^{\lambda},$ with all $a_{\lambda}(q,t) \neq 0$ non-zero $q,t$ rational functions. Therefore, $\{\lambda \mid \lambda \leq \mu \} = \supp(E_{\mu})$, and so the Newton polytope of $E_{\mu}$ is the convex hull of the Bruhat lower order ideal of $\mu$:  $\Newt(E_{\mu}) =  \mathrm{conv}(\lambda \mid \lambda \leq \mu).$
\end{prop}

Theorem \ref{thm:MacdonaldSNP} implies that Conjecture \ref{conj: snp conj} is true.

\begin{cor}\label{cor: proof of conj}
    Conjecture 3.8 \cite{SNPOriginal} is true.
\end{cor}

\begin{proof}[Proof of Corollary \ref{cor: proof of conj}]
Sahi showed in \cite{Sahi} that $\text{supp}(E_{\alpha}) = \{\beta: \beta \leq_{S} \alpha\}$, so by Theorem \ref{thm:MacdonaldSNP},
\[\{\beta: \beta \leq_{S} \alpha\} = \text{supp}(E_{\mu}) = \text{conv}(\text{supp}(E_{\mu})) \cap \mathbb{Z}^{n}_{\geq0} = \text{conv}(\{\beta: \beta \leq_{S} \alpha\}) \cap \mathbb{Z}^{n}_{\geq 0} = \widehat{\mathcal{P}_{\alpha}} \cap \mathbb{Z}^{n}_{\geq 0}. \]
In particular, the second equality is Theorem \ref{thm:MacdonaldSNP} and the rest is immediate by definition. 
\end{proof}


\subsection{Moment polytopes}

In this section, we relate the $M$-convexity of the support sets of non-symmetric Macdonald polynomials to the moment polytopes of certain affine Schubert varieties. 

Here, we give a brief overview of the construction of affine Grassmannians. We refer the reader to \cite{BessonHong} for a more complete introduction to affine Grassmannians and, specifically, to Section 4.4 for their discussion regarding the moment polytopes of affine Schubert varieties.

\begin{definition}
    For a commutative ring $R$, let $\mathrm{GL}_n(R)$ denote the group of invertible $n\times n$ matrices with coefficients in $R.$ Let $B_n^{+}(R)$ denote the subgroup of $\mathrm{GL}_n(R)$ consisting of upper triangular matrices. Let $\mathbb{C}[[t]]$ be the ring of complex $1$-variable formal power series with field of fractions $\mathbb{C}((t)),$ the field of complex $1$-variable formal Laurent series. The affine Grassmannian is the space of cosets $\mathrm{Gr}_n:= \mathrm{GL}_n(\mathbb{C}((t)))/\mathrm{GL}_n(\mathbb{C}[[t]]).$ Define $\rho:\mathrm{GL}_n(\mathbb{C}[[t]]) \rightarrow \mathrm{GL}_n(\mathbb{C})$ to be the group homomorphism determined by $t \mapsto 0.$ The Iwahori subgroup $\mathrm{I}_n$ is defined to be the pre-image of $B_n^{+}(\mathbb{C})$ with respect to $\rho$, i.e., $\mathrm{I}_n:= \rho^{-1}(B_n^{+}(\mathbb{C})).$ For $\mu \in \mathbb{Z}^n$, define the point $L_{\mu} \in \mathrm{Gr}_n$ as $L_{\mu}:= t^{\mu}\mathrm{GL}_n(\mathbb{C}[[t]])/\mathrm{GL}_n(\mathbb{C}[[t]]).$ The open Schubert variety $X_{\mu}$ is defined to be the Iwahori subgroup orbit of the point $L_{\mu}$, that is to say, $X_{\mu}:= \mathrm{I}_n .L_{\mu} \subset \mathrm{Gr}_n.$ The affine Schubert variety $\overline{X_{\mu}}$ is defined as the closure of $X_{\mu}$ in $\mathrm{Gr}_n.$
 \end{definition}

Using the Iwahori decomposition of $\mathrm{GL}_n(\mathbb{C}((t)))$, we know that the points $\{ L_{\mu} \mid \mu \in \mathbb{Z}^n \}$ are exactly the $(\mathbb{C}^*)^n$-fixed points of the affine Grassmannian $\mathrm{Gr}_n.$ Let $T:= (\mathbb{C}^*)^n$ denote this (small) torus. In fact, each open Schubert variety $X_{\mu}$ contains exactly one $T$-fixed point, namely $L_{\mu}.$

Importantly, the partial order $\leq$ on $\mathbb{Z}^n$ (see Definition \ref{def: bruhat}) is the same as the (parabolic) Bruhat order on the set of minimal length coset representatives $\widehat{\mathfrak{S}}_n/\mathfrak{S}_n \equiv \mathbb{Z}^n$ (see \cite{HHLnsym}). Analogous to the usual Schubert cell decomposition for finite dimensional Grassmannian Schubert varieties, the affine Schubert varieties have the following property.

\begin{lem}[Lemma 2.22 \cite{BessonHong}]
    For all $\mu \in \mathbb{Z}^n$, 
    $\overline{X_{\mu}} = \bigsqcup_{\lambda \leq \mu} X_{\lambda}$ where $\lambda \leq \mu$ denotes the Bruhat order on $\mathbb{Z}^n.$
\end{lem}

The moment polytope $\mathrm{MP}(\overline{X}_{\mu})$ of $\overline{X}_{\mu}$ is thus $\mathrm{conv}(\lambda \mid \lambda \leq \mu).$ By Proposition \ref{prop:bruhat interval}, it follows that $\mathrm{MP}(\overline{X}_{\mu}) = \Newt(E_{\mu})$ for generic $q,t.$ From Theorem \ref{thm:MacdonaldSNP}, we know that for generic $q,t$ the support set $\supp(E_{\mu})$ is $M$-convex. Since $\conv(\supp(E_{\mu})) = \Newt(E_{\mu}) = \mathrm{MP}(\overline{X}_{\mu})$, we find by the definition of M-convexity that Corollary \ref{cor: schubert mp} must hold.

\section{Recovering Non-symmetric Macdonald Polynomials from the Dyck Graph}\label{recoveringNonSym}

While we only apply the graph isomorphism $\mathcal{F}$ to understand the supports of non-symmetric Macdonald polynomials, we can rewrite the HHL formula purely in the language of Dyck graphs and colorings. Note this is distinct from the work of Haglund and Wilson in \cite{ChromaticMacdonald}, where they decomposed the various Macdonald polynomials into chromatic quasisymmetric functions and their variants. We are not offering a decomposition but instead a chromatic reinterpretation of the HHL combinatorial formula. Surprisingly, the statistics we define for the family of Dyck paths associated to an attacking graph could be computed for any Dyck graph. Thus, our results suggest an avenue to extend combinatorial aspects of variants of Macdonald polynomials to any Dyck graph. 

In order to describe our Dyck path interpretation of the non-symmetric Macdonald polynomials, we finally define the combinatorial statistics seen in Definition \ref{def: nsym macd}.

\begin{definition}
    Let $\mu$ be a weak composition. Given $u = (i,j) \in dg'(\mu)$ define the following: 
\begin{itemize}
    \item $\mathrm{leg}(u) := \{(i,j') \in dg'(\mu): j' > j\}$
    \item $\mathrm{arm}^{\text{left}}(u) := \{(i',j) \in dg'(\mu): i'<i, \mu_{i'} \leq \mu_i\} $
    \item $\mathrm{arm}^{\text{right}}(u):= \{(i',j-1) \in \widehat{dg}(\mu): i'>i, \mu_{i'}<\mu_{i}\}$
    \item $\mathrm{arm}(u) := \mathrm{arm}^{\text{left}}(u) \cup \mathrm{arm}^{\text{right}}(u)$
    \item $l(u):= |\mathrm{leg}(u)| = \mu_i -j$
    \item $a(u) := |\mathrm{arm}(u)|.$
\end{itemize}

Let $d(u) = (i,j-1)$ denote the box just below $u$. Given a filling $\sigma:dg'(\mu)\rightarrow \{1,\dots,n\}$, a descent of $\sigma$ is a box $u \in dg'(\mu)$ such that $\widehat{\sigma}(u) > \widehat{\sigma}(d(u)).$
Set $\mathrm{Des}(\widehat{\sigma})$ to be the set of descents of $\widehat{\sigma}$. The reading order on the diagram $\widehat{dg}(\mu)$ is the total ordering on the boxes of $\widehat{dg}(\mu)$ row by row, from top to bottom, and from right to left within each row. If $\sigma: dg'(\mu) \rightarrow \{1,\dots,n\}$ is a filling, an inversion of $\widehat{\sigma}$ is a pair of attacking boxes $u,v \in \widehat{dg}(\mu)$ such that $u < v$ in reading order and $\widehat{\sigma}(u) > \widehat{\sigma}(v).$ Set $\mathrm{Inv}(\widehat{\sigma})$ to be the set of inversions of $\widehat{\sigma}$. 

Define the statistics
    \begin{itemize}
        \item $\mathrm{maj}(\widehat{\sigma}):= \sum_{u \in \mathrm{Des}(\widehat{\sigma})}(l(u) +1)$
        \item $\mathrm{inv}(\widehat{\sigma}):= |\mathrm{Inv}(\widehat{\sigma})|
        -|\{i<j|  \mu_i\leq \mu_j\}| - \sum_{u \in \mathrm{Des}(\widehat{\sigma})} a(u)$
        \item $\mathrm{coinv}(\widehat{\sigma}):= \left(\sum_{u \in dg'(\mu)} a(u) \right) - \mathrm{inv}(\widehat{\sigma}).$
    \end{itemize}
\end{definition}

For example, consider the following non-attacking filling $\widehat{\sigma}$:
\[\begin{ytableau}
        \none & 2 & \none \\
        \none & 1 & 3 \\
        1 & 2 & 3 \\ \end{ytableau}\]
The monomial associated to it $x_{1}x_{2}x_{3}$, since there is $1$ of each label above the basement. We read the entries from left to right and bottom to top labeling them $1,2,3,4,5,6$. To compute each of the statistics, we start with $\text{Des}(\widehat{\sigma})$. Note that $6$ is the unique descent as its label $2$ is greater than the label $1$ of the box below it. Then $\text{maj}(\widehat{\sigma})$ is $l(6) + 1 = |\text{leg}(6)|+1$. By definition, leg is the set of boxes above $6$ in the same column. There is no such box, hence, $l(6) = 0$ and $\text{maj}(\widehat{\sigma}) = 1$. The arm of an element is its set of neighbors with smaller index. Furthermore, $\text{arm}(6) = \text{arm}^{\text{left}}(6) \sqcup \text{arm}^{\text{right}}(6)$. Left arm is the set of neighbors in the same row to the left in a weakly shorter column and so $\text{arm}^{\text{left}}(6) =\emptyset$. Right arm is the set of neighbors in the lower row to the right in a strictly shorter column, so $\text{arm}^{\text{right}}(6) = 1$. Hence, $a(6) = |\text{arm}(6)| = 1$. More generally, $a(4) = 1$, and $a(5) = 0$. Then 
\[\sum_{u \in \text{dg'}(\mu)} a(u) = a(4)  + a(6) =  2 \text{ and } \sum_{u \in \text{Des}(\widehat{\sigma})} a(u) = a(6) = 1.\]
Note that $|\{i < j: \mu_{i} \leq \mu_{j}\}|$ is the set of pairs of columns such that the earlier column is shorter than the latter. In this case, $\mu = (0,2,1)$, so 
\[|\{i < j: \mu_{i} \leq \mu_{j}\}| = 2.\]
The set of inversions is the set of pairs of adjacent elements in the filling of the augmented diagram in the correct order for our reading order, so $\{(1,2), (1,3), (2,3), (4,5)\}$ are the inversions meaning that $|\text{Inv}(\widehat{\sigma})| = 4$. It follows that
\[\text{inv}(\widehat{\sigma}) = 4 - 2 - 1 = 1,\]
and $\text{coinv}(\widehat{\sigma}) = 2 - 1 = 0$. Note that only boxes $4$ and $6$ are not equal to the boxes below them. Box $4$ has leg $1$, so putting all of this together yields that coefficient of the monomial is 
\[q^{\mathrm{maj}(\widehat{\sigma})}t^{\mathrm{coinv}(\widehat{\sigma})} \prod_{\substack{\square \in \dg'(\mu) \\ \widehat{\sigma}(\square) \neq \widehat{\sigma}(d(\square))}} \left( \frac{1-t}{1-q^{l(\square)+1}t^{a(\square)+1}} \right) = qt\left( \frac{1-t}{1-q^{2}t^{2}} \right)\left(\frac{1-t}{1-qt^{2}}\right).\]

Now we describe combinatorial statistics for Dyck graphs which generalize the above statistics for non-symmetric Macdonald polynomials with respect to the bijection $\mathcal{F}.$

Let $\delta$ be a Dyck path with $k$ initial north steps. We define a pre-coloring $c_{\delta}:\{1,\dots,k\} \rightarrow \{1,\dots,k\}$ of the vertices $\{1,\dots, k\}$ of the Dyck graph $G_{\delta}$ via $c_{\delta}(i):=i.$ 
    Define the column graph of $\delta,$ $C_{\delta}$, to be the directed graph on the vertex set $\{1,\dots,n\}$ whose directed edges $j \rightarrow i$ are those pairs of $i<j$ such that $i$ is the largest number such that $\{i,j\}$ does not form an edge in $G_{\delta}.$ We write $d(j):= i$ if $j \rightarrow i$ form an edge of $C_{\delta}$. An equivalent way to find this is to start at $(j,j)$ and move to the left until hitting the Dyck path at $(i,j)$. Then if $i > 0$, $d(j) = i$. Otherwise, $j$ has no outgoing neighbor.

    The leg set of $k+1 \leq i \leq n$, $\mathrm{leg}(i)$, is the set of all $i< j \leq n$ such that these exists a path from $j$ to $i$ in the column graph $C_{\delta}.$ Write $l(i):= |\mathrm{leg}(i)|$ for the size of the leg set of $i.$ The vertices $\{1,\dots , k\}$ are exactly the set of sinks for the column graph $C_{\delta}$ so we may partition the vertices $\{1,\dots ,k\}$ into the components of the graph $C_{\delta}$ to obtain corresponding sets $C_1,\dots,C_k$ which we call the columns of $\delta.$ Define the column height function $\mathrm{ht}: \{1,\cdots, n\}\rightarrow \mathbb{N}$ by $\mathrm{ht}(i):= |C_j|$ if $i \in C_j.$ Let $b_{\delta}$ denote the bounce path of $\delta$ formed by starting at the bottom left vertex of the Dyck path $\delta$ and successively moving north and east maximally while remaining between the path $\delta$ and the main diagonal. The bounce path $b_{\delta}$ partitions the set $\{1,\dots,n\}$ into subsets $R_1,\dots,R_r$ where each $R_i$ is a maximal clique of the Dyck graph of $b_{\delta}.$ We call the sets $R_i$ the rows of $\delta.$ 
    
    For $1\leq i\leq n$ define the arm set, $\mathrm{arm}(i)$, to consist of all $1 \leq j < i$ such that either 
    \begin{itemize}
        \item $i$ and $j$ are in the same row of $G_{\delta}$ and $\mathrm{ht}(j) \leq \mathrm{ht}(i)$
        \item $i$ and $j$ are connected in $G_{\delta}$ but not in the same row and $\mathrm{ht}(j) < \mathrm{ht}(i).$
    \end{itemize}
    
    We write $a(i):= |\mathrm{arm}(i)|.$ Suppose $f: \{1,\dots,n\} \rightarrow \mathbb{N}$ is a labeling of the vertices of the Dyck graph of $\delta$. We say that $1 \leq j \leq n$ is a descent of $f$ if $f(j) > f(d(j)).$ Write $\mathrm{Des}(f)$ for the set of descents of $f.$ An inversion of $f$ is a pair $i<j$ with $\{i,j\}$ an edge in $G_{\delta}$ such that $f(i) < f(j)$. We write $\mathrm{Inv}(f)$ for the set of all inversion pairs of $f.$ Define the statistics
    \begin{itemize}
        \item $\mathrm{maj}(f):= \sum_{i \in \mathrm{Des}(f)}(l(i) +1)$
        \item $\mathrm{inv}(f):= |\mathrm{Inv}(f)|
        -|\{i<j|  \mathrm{ht}(i)\leq \mathrm{ht}(j)\}| - \sum_{i \in \mathrm{Des}(f)} a(i)$
        \item $\mathrm{inv}^*(f):= |\mathrm{Inv}(f)|
         - \sum_{i \in \mathrm{Des}(f)} a(i)$
        \item $\mathrm{coinv}(f):= \left(\sum_{i \geq k+1} a(i) \right) - \mathrm{inv}(f).$
    \end{itemize}

\begin{definition}
    Let $\delta$ be a Dyck path with $k$ initial north steps. Define the Dyck path Macdonald polynomial of $\delta$, $\Psi_{\delta} = \Psi_{\delta}(x_1,\dots,x_k;q,t)\in \mathbb{Q}(q,t)[x_1,\dots,x_k]$, as
    $$\Psi_{\delta}:= \sum_{f\in C_{k}(G_{\delta},c_{\delta})} x^{f} q^{\mathrm{maj}(f)}t^{\mathrm{coinv(f)}} \prod_{\substack{i\geq k+1\\ f(i) \neq f(d(i))}} \frac{1-t}{1-q^{l(i)+1}t^{a(i)+1}}.$$
\end{definition}

The Dyck path Macdonald polynomials generalize the non-symmetric Macdonald polynomials.

\begin{proposition}
    For any weak composition $\mu$, $E_{\mu} = \Psi_{\mathcal{F}(\mu)}.$
\end{proposition}
\begin{proof}
   This follows from aligning the combinatorics of the HHL formula for non-symmetric Macdonald polynomials with the Dyck path combinatorics described above. This is straightforward to verify using the Dyckification bijection in Theorem \ref{thm:Dyckification}.
\end{proof}

Using the HHL combinatorial formula for the non-symmetric Macdonald $E_{0,2,1}(x_1,x_2,x_3;q,t)$, we saw that the above non-attacking filling contributes the term 
$x_1x_2x_3 qt \frac{1-t}{1-q^2t^2} \frac{1-t}{1-qt^2}.$
Under the Dyckification map, this non-attacking filling yields the following Dyck path and coloring on the left with the unique descent entry bolded and column graph on the right:
\[   
    \begin{tikzpicture}
    \draw[red, thick, dashed] (0,0) -- (0,1)-- (0,2) -- (0,3) -- (1,3) -- (2,3) -- (2,4) -- (3,4) -- (3,5) -- (4,5) -- (4,6) -- (5,6) -- (6,6);
    \draw (.5,.5) node {$1$};
    \draw (1.5,1.5) node {$2$};
    \draw (2.5,2.5) node {$3$};
    \draw (3.5,3.5) node {$1$};
    \draw (4.5,4.5) node {$3$};
    \draw (5.5,5.5) node {$\mathbf{2}$};
    
  \foreach \x in {0,1,2,3,4,5,6} {
    \foreach \y in {0,1,2,3,4,5,6} {
      \fill (\x, \y) circle (2pt); 
    }
  }
\end{tikzpicture}
 \hspace{1 in} 
\begin{tikzpicture}
    \draw[red, thick, dashed] (0,0) -- (0,1)-- (0,2) -- (0,3) -- (1,3) -- (2,3) -- (2,4) -- (3,4) -- (3,5) -- (4,5) -- (4,6) -- (5,6) -- (6,6);
    \draw (.5,.5) node {$1$};
    \draw (1.5,1.5) node {$2$};
    \draw (2.5,2.5) node {$3$};
    \draw (3.5,3.5) node {$4$};
    \draw (4.5,4.5) node {$5$};
    \draw (5.5,5.5) node {$6$};
    \draw[blue, thick, ->] (3.3,3.5) .. controls (1.5,3.5) .. (1.5,1.7);
    \draw[blue, thick, ->] (4.3,4.5) .. controls (2.5,4.5) .. (2.5,2.7);
    \draw[blue, thick, ->] (5.3,5.5) .. controls (3.5,5.5) .. (3.5,3.7);
  \foreach \x in {0,1,2,3,4,5,6} {
    \foreach \y in {0,1,2,3,4,5,6} {
      \fill (\x, \y) circle (2pt); 
    }
  }
\end{tikzpicture}
\]
Below on the left is the Dyck path with the unique non-trivial inversion, and on the right the bounce path of the Dyck path is depicted.
\[ \begin{tikzpicture}
    \draw[red, thick, dashed] (0,0) -- (0,1)-- (0,2) -- (0,3) -- (1,3) -- (2,3) -- (2,4) -- (3,4) -- (3,5) -- (4,5) -- (4,6) -- (5,6) -- (6,6);
    \draw (.5,.5) node {$1$};
    \draw (1.5,1.5) node {$2$};
    \draw (2.5,2.5) node {$3$};
    \draw (3.5,3.5) node {$1$};
    \draw (4.5,4.5) node {$3$};
    \draw (5.5,5.5) node {$2$};
    
    \draw[blue, dotted, thick] (3.5,3.7) -- (3.5,4.5);
    \draw[blue, dotted, thick] (3.5,4.5) -- (4.3, 4.5);
  \foreach \x in {0,1,2,3,4,5,6} {
    \foreach \y in {0,1,2,3,4,5,6} {
      \fill (\x, \y) circle (2pt); 
    }
  }
\end{tikzpicture} 
 \hspace{1 in}
 \begin{tikzpicture}
    \draw (.5,.5) node {$1$};
    \draw (1.5,1.5) node {$2$};
    \draw (2.5,2.5) node {$3$};
    \draw (3.5,3.5) node {$4$};
    \draw (4.5,4.5) node {$5$};
    \draw (5.5,5.5) node {$6$};

    \draw[blue, thick, dotted] (0,0) -- (0,3) -- (3,3) -- (3,5) -- (5,5)--(5,6) -- (6,6);
    
  \foreach \x in {0,1,2,3,4,5,6} {
    \foreach \y in {0,1,2,3,4,5,6} {
      \fill (\x, \y) circle (2pt); 
    }
  }
\end{tikzpicture}\]

As such the row sets of $\mathcal{F}(0,2,1)$ are $R_1 = \{1,2,3\}, R_2 = \{4,5\}$ and $R_3 = \{6\}$ whereas the column sets are $C_1 = \{1\}, C_2=\{2,4,6\},$ and $C_3=\{3,5\}.$ Note that these are the row and column sets of the augmented diagram of $(0,2,1)$. The arm and leg sets are as follows: $\mathrm{arm}(4) = \{3\}, \mathrm{arm}(5) = \emptyset, \mathrm{arm}(6) = \{5\}, \mathrm{leg}(4) = \{6\},\mathrm{leg}(5) = \emptyset,$ and $\mathrm{leg}(6) = \emptyset.$ The inversion set of $f$ is $\mathrm{Inv}(f) = \{ (1,2),(1,3),(2,3),(4,5)\}$ and the descent set is $\mathrm{Des}(f) = \{6\}.$ Thus this filling contributes the term $x_1x_2x_3 qt \frac{1-t}{1-q^2t^2} \frac{1-t}{1-qt^2}$ to the polynomial $\Psi_{NNNEENENENEE}(x_1,x_2,x_2;q,t)$ as expected. 

In this work, our main application for the interpretation in terms of Dyck graph colorings for non-symmetric Macdonald polynomials is to prove saturation. However, this framework extends beyond that result. Modified Macdonald functions appeared recently in connection to chromatic symmetric functions \cite{macdonaldstanleystembridge, ShuffleConjecture, ChromaticMacdonald}, but this connection remains mysterious. In pursuit of better understanding this connection, we also provide a reinterpretation of the HHL formula for modified Macdonald functions from Theorem 5.1.1 in \cite{HHLnsym} using (non-proper) colorings of Dyck graphs. 

\begin{definition}
    Let $\delta$ be a Dyck path. Define the modified Dyck path Macdonald function of $\delta$, $\Phi_{\delta} = \Phi_{\delta}(x_1,x_2,\dots)$, by 
    $$\Phi_{\delta}:= \sum_{f: G_{\delta} \rightarrow \mathbb{N}} q^{\mathrm{maj}(f)}t^{\mathrm{inv}^*(f)}x^f .$$
\end{definition}

For any weak composition $\mu$ we define $\mathrm{sort}(\mu)$ to be the partition whose parts consist of the nonzero elements of $\mu$ in weakly decreasing order. For the composition $\mu \in \mathbb{Z}_{\geq 0}^{n}$ with corresponding Dyck path $\mathcal{F}(\mu)$ we define $\mathcal{F}(\mu)'$ to be the Dyck path formed by considering the induced subgraph of $G_{\mathcal{F}(\mu)}$ corresponding to the vertices $\{n+1,\dots,n+|\mu|\}$ which we re-label as $\{1,\dots, |\mu|\}$. This has the effect of forming a new Dyck path which corresponds to diagram of $\mu$ where the basement of the augmented diagram has been removed. The following proposition is straightforward to verify.

\begin{proposition}
    For any weak composition $\mu$, $\widetilde{H}_{\mathrm{sort}(\mu)} = \Phi_{\mathcal{F}(\mu)'}.$
\end{proposition}

A similar approach also gives an interpretation for symmetric Macdonald polynomials by only allowing proper colorings and permuted-basement Macdonald polynomials by taking the non-symmetric case and allowing any fixed proper coloring of vertices under the first bounce in the bounce path. 

Note that, as defined, we could apply the definition to any Dyck graph and arrive at a generalization of modified Macdonald functions for any Dyck graph. It turns out that, in general, the power series $\Phi_{\delta}$ is not always a symmetric function like in the modified Macdonald function case, but it is always quasi-symmetric.

\begin{definition}
    A formal power series $f(x_1,x_2,x_3,\dots)$ is \textbf{quasi-symmetric} if the coefficient of any monomial $x_1^{\alpha_1}\cdots x_{k}^{\alpha_k}$ in $f$ is the same as the coefficient of $x_{i_1}^{\alpha_1}\cdots x_{i_k}^{\alpha_k}$ in $f$ for all $i_1< \dots < i_k.$ For any weak composition $\alpha = (\alpha_1,\dots,\alpha_k)$ let $M_{\alpha}:= \sum_{i_1< \dots < i_k} x_{i_1}^{\alpha_1}\cdots x_{i_k}^{\alpha_k}$ denote the \textbf{monomial quasi-symmetric function}. Define the \textbf{Gessel fundamental quasi-symmetric functions}, $Q_{n,D}$, for $D \subseteq 2^{[n-1]}$ by  
    $$ Q_{n,D} = \sum_{\substack{i_1\leq \dots \leq i_n \\ i_{j} = i_{j+1} \Rightarrow j \in D}} x_{i_1}\cdots x_{i_n}.$$ 
\end{definition}

\begin{proposition}\label{prop:quasisym}
    For all Dyck paths $\delta,$ $\Phi_{\delta}(x_1,x_2,x_3,\dots)$ is a quasi-symmetric function.
\end{proposition}
\begin{proof}
    For any fixed coloring of the Dyck graph, changing the labels of the colors without changing the relative ordering yields another valid coloring. Furthermore, $\text{maj}$ and $\text{inv}^{\ast}$ are invariant under such a change. 
\end{proof}

It is clear from Proposition \ref{prop:quasisym} that for all Dyck paths $\delta$ of length $n$, $\Phi_{\delta} \in \mathbb{Z}_{\geq 0}[q,t^{\pm 1}]\{ M_{\alpha}\}_{\alpha \models n}.$ Upon specializing $q=t=1$, we immediately see that $\Phi_{\delta}|_{q=t=1} = (x_1+x_2+\dots)^n.$ More interestingly, by setting $q=0$ we obtain
$$\Phi_{\delta}|_{q=0} = \sum_{\substack{f: G_{\delta} \rightarrow \mathbb{N}\\ \mathrm{Des}(f) = \emptyset}}t^{|\mathrm{Inv}(f)|}x^f.$$ The condition $\mathrm{Des}(f) = \emptyset$ is equivalent to the statement that whenever $j \rightarrow i$ in the column graph $C_{\delta}$ we have that $f(j) \leq f(i).$ If we further specialize $t= 0$, we obtain $\Phi_{\delta}|_{q=t=0}=h_n = \sum_{i_1\leq \dots \leq i_n} x_{i_1}\cdots x_{i_{n}}$ and if we specialize instead $t \mapsto 1$ we get 
$$\Phi_{\delta}|_{q=0,t=1} = \sum_{\substack{f: G_{\delta} \rightarrow \mathbb{N}\\ \mathrm{Des}(f) = \emptyset}}x^f.$$ This latter specialization is an example of a \textbf{P-partition generating function}, namely it is the generating function corresponding to the poset on $\{1,\dots,n\}$ generated by the cover relation $i \prec j$ if $i = d(j)$ with the usual vertex ordering. 

Consider next the following example of a Dyck graph. This is the minimal example that is not in the image of the Dyckification map of weak compositions:
\[ \begin{tikzpicture}
    \draw[red, thick, dashed] (0,0) -- (0,1)-- (1,1) -- (1,2) -- (1,3) -- (2,3) -- (3,3);
    \draw (.5,.5) node {$1$};
    \draw (1.5,1.5) node {$2$};
    \draw (2.5,2.5) node {$3$};
    
  \foreach \x in {0,1,2,3} {
    \foreach \y in {0,1,2,3} {
      \fill (\x, \y) circle (2pt); 
    }
  }
\end{tikzpicture}\]
Now this Dyck graph has no non-attacking colorings in the non-symmetric model as the number of colors allowed is precisely the number of initial north steps of the Dyck graph. In this case, there is precisely one north step, but the graph is not 1-colorable meaning there is no proper coloring. However, in the modified model, the number of colors is unbounded, but quasi-symmetry implies that considering at most three colors is sufficient. We find that $\Phi_{\delta}(x_1,x_2,x_3;q,t)$ is given by
$$ (x_1^3+x_2^3+x_3^3) + (1+2q)(x_1^2x_2 + x_1^2x_3+x_2^2x_3) + (1+t+q^2t^{-1})(x_1x_2^2+x_1x_3^2+x_2x_3^2) + (1+2q+t+q^2+q^2t^{-1})x_1x_2x_3.$$ Therefore, the full series is given by 
$$\Phi_{\delta} = M_3 + (1+2q)M_{2,1}+(1+t+q^2t^{-1})M_{1,2} + (1+2q+t+q^2+q^2t^{-1})M_{1,1,1}.$$ This shows that $\Phi_{\delta} = Q_{3,\emptyset} + 2qQ_{3,\{2\}}+(t+q^2t^{-1})Q_{3,\{1\}}+q^2Q_{3,\{1,2\}}$. Thus we find that $\Phi_{\delta} \in \mathbb{Z}_{>0}[q,t^{\pm 1}]\{ Q_{3,D}\}_{ D \in 2^{ \{1,2\}}}$. Furthermore, each $Q_{3,D}$ coefficient in $\Phi_{\delta}$ is nonzero. This is true in general.

\begin{proposition}
    For all Dyck paths $\delta$ of length $n,$  $\Phi_{\delta} \in \mathbb{Z}_{>0}[q,t^{\pm 1}]\{ Q_{n,D}\}_{ D \in 2^{[n-1]}}.$
\end{proposition}
\begin{proof}
    This follows from a standardization argument very similar to the derivation of the Gessel expansions of the modified Macdonald functions in \cite{HHLsym}. Let $f: G_{\delta} \rightarrow \mathbb{N}$. We may consider $f$ as a filling of the diagonal boxes of the Dyck path $\delta.$
    
    That is, recall that we denote the vertices of the Dyck graph by $[n]$, so $f: [n] \to [n]$. We define the standardization, $w_{f} \in \mathfrak{S}_n$ by constructing a total ordering $\prec$ on the boxes as follows. For $1 \leq i < j \leq n$, we say $i \prec j$ if $f(i) < f(j)$ and $i \succ j$ otherwise. By construction, this is a total ordering as any two elements are comparable, and it is well-defined by virtue of $f$ being well-defined. Furthermore, also by construction,  $w_{f}$ has the same inversion set as the word $f$. In other words, for all $1 \leq i < j \leq n$, $i \prec j$ if and only if $f(i) < f(j)$. The statistics $\text{maj}$ and $\text{inv}^{\ast}$ depend exclusively on the inversion set, and so $f$ has the same $\text{maj}$ and $\text{inv}^{\ast}$ values as $w_{f}$. We write $\text{maj}_{\delta}(w)$ and $\text{inv}^{\ast}(w)$ for $w \in \mathfrak{S}_n.$
    
    Each permutation $w\in \mathfrak{S}_n$ determines a Gessel function $Q_{n,D^{-1}(w)}$, where $D^{-1}(w) = \{i \in [n]: w(i) > w(i+1)\}$ is the descent set of the inverse of $w$. A standard argument as in \cite[Proposition 4.3]{HHLsym} yields that 
    \[Q_{n,D^{-1}(w)} = \sum_{\substack{f : [n] \to [n]\\ w_{f} = w}} x^{f}.\]
    Therefore, we have
    \[\Phi_{\delta} = \sum_{w \in \mathfrak{S}_{n}} q^{\text{maj}_{\delta}(w)}t^{\text{inv}_{\delta}^{\ast}(w)}Q_{n,D^{-1}(w)}.\] The result follows from noting that by definition $\text{maj}(w) \geq 0.$ 
    
    
\end{proof}

The modified Macdonald functions are known to be the doubly-graded Frobenius characteristics of certain $\mathfrak{S}_n$ modules relating to the geometry of Hilbert schemes \cite{haiman2000hilbert}. In particular, the modified Macdonald functions are Schur positive. The fundamental quasi-symmetric functions correspond to the irreducible representations of $0$-Hecke algebras in a manner analogous to the Frobenius characteristic map for symmetric groups. It would be interesting to see if the functions $\Phi_{\delta}$ have a geometric interpretation or if they correspond to some family of understandable doubly-graded $0$-Hecke algebra representations. 

Notably, as the previous example demonstrated, there are $\delta$ with $\Phi_{\delta} \notin \mathbb{Z}_{\geq 0}[q,t]\{ M_{\alpha}\}_{\alpha \models n}.$ This is different from the usual modified Macdonald case, where the parameters $q,t$ correspond geometrically to equivariant K-theoretic parameters of the algebraic torus $\mathbb{C}^* \times \mathbb{C}^*$ \cite{haiman2000hilbert}. In that case, the relevant representations of that torus are polynomial, so only non-negative powers of $q$ and $t$ appear in the character $\widetilde{H}_{\mu}$. However, if one were to consider a geometric situation where the underlying torus representations were simply algebraic (i.e., rational representations are allowed) then one would see characters in $\mathbb{Z}_{\geq 0}[q^{\pm 1}, t^{\pm 1}]$. That is to say, having negative $t$ exponents would not apriori rule out a possible geometric interpretation for the $\Phi_{\delta}.$

One possible method to understand the functions $\Phi_{\delta}$ is the following $3$-parameter generalization. For a filling $f:G_{\delta} \rightarrow \mathbb{N}$ define $\mathrm{amaj}(f):= \sum_{i\in \mathrm{Des}(f)} a(i).$ Define $$\Pi_{\delta}(x_1,x_2,\dots;q,t,u):= \sum_{f: G_{\delta} \rightarrow \mathbb{N}} q^{\mathrm{maj}(f)}t^{|\mathrm{Inv}(f)|} u^{\mathrm{amaj(f)}}x^f.$$ By definition, we have the specialization $\Phi_{\delta} = \Pi_{\delta}|_{u=t^{-1}}.$ Using analogous arguments to the case of $\Phi_{\delta}$, we see that $\Pi_{\delta} \in \mathbb{Z}_{\geq 0}[q,t,u]\{ Q_{n,D}\}_{ D \in 2^{[n-1]}}.$ Consider the specialization $\Pi_{\delta}|_{q=u=1}.$ This yields 
    $$\Pi_{\delta}|_{q=u=1} = \sum_{f: G_{\delta} \rightarrow \mathbb{N}} t^{|\mathrm{Inv}(f)|} x^f.$$ Due to a result of Carlsson and Mellit \cite{ShuffleConjecture}, $\Pi_{\delta}|_{q=u=1}$ is a symmetric function related to the \textbf{chromatic quasi-symmetric function} of the Dyck path $\delta$, $\chi_{\delta}(X;t)$, via the plethystic transformation 
    $$\Pi_{\delta}(X;1,t,1) = (t-1)^{n} \chi_{\delta}\left( \frac{X}{t-1};t\right).$$ Thus we see that $\Pi_{\delta}$ simultaneously generalizes (up to a plethystic transformation) the modified Macdonald polynomials, the P-partition generating functions for unit interval orders, and the quasi-symmetric chromatic functions of Dyck paths. It is not clear to these authors whether the functions $\Pi_{\delta}$,$\Phi_{\delta}$ are examples of the weighted characteristic functions $\chi(\delta,\mathrm{wt})$ of Carlsson-Mellit for some choice of weighting $\mathrm{wt}$. We leave studying the polynomials $\Psi_{\delta},\Phi_{\delta}$, and $ \Pi_{\delta}$ in greater generality as a further research direction. 

    

\bibliographystyle{amsplain}
\bibliography{bibliography.bib}

\end{document}